\documentclass{article}%

\usepackage{amsmath,amssymb,amsfonts}%
\usepackage{theorem}%
\usepackage{graphicx}%
\usepackage[arrow,matrix,tips]{xy}%
\usepackage{hyperref}%

\theoremstyle{change}%

\newtheorem{definition}{Definition:}[section]%
\newtheorem{theorem}[definition]{Theorem:}%
\newtheorem{proposition}[definition]{Proposition:}%
\newtheorem{lemma}[definition]{Lemma:}%
\newtheorem{corollary}[definition]{Corollary:}%
{\theorembodyfont{\rmfamily} \newtheorem{remark}[definition]{Remark:}}%
{\theorembodyfont{\rmfamily} \newtheorem{example}[definition]{Example:}}%

\newenvironment{proof}
  {{\bf Proof:}}
  {\qquad \hspace*{\fill} $\Box$}%

\newcommand{\id}{\operatorname{id}}%
\newcommand{\inner}{\operatorname{int}}%
\newcommand{\tp}{\operatorname{top}}%
\newcommand{\sep}{\operatorname{sep}}%
\newcommand{\spn}{\operatorname{span}}%
\newcommand{\diam}{\operatorname{diam}}%
\newcommand{\dist}{\operatorname{dist}}%
\newcommand{\vB}{\mathbf{v}}
\newcommand{\wB}{\mathbf{w}}
\newcommand{\Graph}{\operatorname{Graph}}%

\newcommand{\AC}{\mathcal{A}}%
\newcommand{\EC}{\mathcal{E}}%
\newcommand{\LC}{\mathcal{L}}%
\newcommand{\NC}{\mathcal{N}}%
\newcommand{\PC}{\mathcal{P}}%
\newcommand{\QC}{\mathcal{Q}}%
\newcommand{\RC}{\mathcal{R}}%
\newcommand{\UC}{\mathcal{U}}%
\newcommand{\WC}{\mathcal{W}}%
\newcommand{\N}{\mathbb{N}}%
\newcommand{\R}{\mathbb{R}}%
\newcommand{\C}{\mathbb{C}}%
\newcommand{\ep}{\varepsilon}%
\newcommand{\tm}{\times}%
\newcommand{\rmS}{\mathrm{S}}%
\newcommand{\Mis}{\mathrm{M}}%

\begin{document}

\title{Some results on the entropy of nonautonomous dynamical systems}%
\author{Christoph Kawan\thanks{Partly supported by DFG fellowship KA 3893/1-1. CK thanks Tomasz Downarowicz and Sergiy Kolyada for helpful comments.}\\
Universit\"at Passau, Fakult\"at f\"ur Informatik und Mathematik\\ Innstra{\ss}e 33, 94032 Passau, Germany\\
christoph.kawan@uni-passau.de
 \and Yuri Latushkin\thanks{Supported by the NSF grant  DMS-1067929, by the Research Board and Research Council of the University of Missouri, and by the Simons Foundation. YL sincerely thanks Lai-Sang Young for the opportunity to spend the academic year 2014/15 at the Courant Institute of Mathematical Sciences and for many illuminating discussions.}\\
 Department of Mathematics, University of Missouri\\
 Columbia, MO 65211, USA\\
 latushkiny@missouri.edu}%
\maketitle%

\begin{abstract}
In this paper we advance the entropy theory of discrete nonautonomous dynamical systems that was initiated by Kolyada and Snoha in 1996. The first part of the paper is devoted to the measure-theoretic entropy theory of general topological systems. We derive several conditions guaranteeing that an initial probability measure, when pushed forward by the system, produces an invariant measure sequence whose entropy captures the dynamics on arbitrarily fine scales. In the second part of the paper, we apply the general theory to the nonstationary subshifts of finite type, introduced by Fisher and Arnoux. In particular, we give sufficient conditions for the variational principle, relating the topological and measure-theoretic entropy, to hold.%
\end{abstract}


\section{Introduction}%

In recent years, properties of discrete nonautonomous dynamical systems (also called \emph{nonautonomous discrete systems}, \emph{sequential/nonstationary/time-dependent dynamical systems} or \emph{mapping families}) have been studied a great deal. While a discrete autonomous or classical dynamical system is given by the iterations of a single map $f:X \rightarrow X$, the dynamics of a nonautonomous system is generated by compositions of different maps. Given a sequence $f_n:X \rightarrow X$, $n=0,1,2,\ldots$, the trajectory of an initial point $x_0 \in X$ is the sequence defined by $x_n = f_{n-1}(x_{n-1})$. Without severe restrictions on the sequence $(f_n)_{n=0}^{\infty}$ such as periodicity or stationarity with respect to some probability distribution, it is intuitively clear that the range of phenomena to be observed in such systems is largely broader than what can be seen in autonomous dynamics. Many concepts used intensively in classical dynamics such as periodicity and recurrence do not seem to make sense for nonautonomous systems, while other concepts such as correlation decay, hyperbolicity or entropy can be generalized and to some extent are still useful to describe the properties of such systems (see \cite{AFi,Fis,KSn,KMS,LYo,OSY}, for instance).%

In this paper, we concentrate on the study of metric (i.e., measure-theoretic) and topological entropy. For a sequence of continuous maps on a compact space, topological entropy was introduced in \cite{KSn} and generalized to systems with time-varying state space (i.e., $f_n:X_n\rightarrow X_{n+1}$) in \cite{KMS}. Measure-theoretic notions of entropy were first established in \cite{Can,Ka1}. While the notion introduced in \cite{Can} still assumes a stationary measure, which for a nonautonomous system only exists under very restrictive assumptions, the notion in \cite{Ka1} is completely general, though not canonical, since it depends on a given class of sequences of measurable partitions satisfying certain axioms. However, for a topological system, given by an equicontinuous sequence of maps $f_n:X_n\rightarrow X_{n+1}$ between compact metric spaces, a canonical class can be defined and the associated metric entropy of any invariant sequence of Borel probability measures on the spaces $X_n$ does not exceed the topological entropy (cf.~\cite[Thm.~28]{Ka1}). This establishes one (usually considered the easier) part of the variational principle. The question for which class of systems the inequality between metric and topological entropy can be extended to a full variational principle is still far from an answer. A first obstacle to a generalization of the classical proofs is that in general it is not known whether the allowed class of measurable partitions contains elements of arbitrarily small diameters.%

In the first part of the paper, we thus aim at a characterization of those invariant measure sequences (shortly IMS) which allow for arbitrarily fine scales in the computation of their entropy (called \emph{fine-scale IMS}). We are able to provide several sufficient conditions, which however all have their limitations, meaning that they are far from being necessary. In fact, we have to leave the question, whether every IMS (under obviously necessary conditions) has such a property, open. Nevertheless, we can show that a system has at least one fine-scale IMS if and only if the sequence $(X_n)_{n=0}^{\infty}$ of compact metric spaces is uniformly totally bounded. Furthermore, we prove a product theorem for the entropy of systems with fine-scale IMS and we show that our notion of metric entropy indeed generalizes the classical notion of Kolmogorov and Sinai.%

The second part of the paper is devoted to the computation of metric and topological entropy of nonstationary subshifts of finite type (shortly, NSFT's), that were first investigated in \cite{AFi,Fis} with the motivation to study uniformly hyperbolic mapping sequences via coding and to deduce properties of adic transformations, respectively. While a classical subshift of finite type is given by a finite alphabet $\AC$ and a transition matrix $L$, in the nonstationary case time-varying alphabets and matrices, i.e., sequences $(\AC_n)_{n=0}^{\infty}$ and $(L_n)_{n=0}^{\infty}$, are allowed. The dynamics is still given by the left shift operator, but now the domain of this operator is no longer invariant but also varies in time. That is, we have a sequence of maps between different sequence spaces, and each map is a restriction of the shift. Putting the word metric on each of the spaces, we obtain a formula for the topological entropy that is completely analogous to the classical one, namely%
\begin{equation*}
  h_{\tp} = \limsup_{n\rightarrow\infty}\frac{1}{n}\log\|L_0 L_1 \cdots L_{n-1}\|,%
\end{equation*}
which reduces to the $\log$ of the spectral radius of $L$ in the stationary case. For the metric entropy we can show that every IMS of an NSFT is a fine-scale IMS. In the explicit computation we concentrate on a special IMS introduced by Fisher that he constructed analogously to the Parry measure in the classical theory. While the Parry measure is the unique measure of maximal entropy for a stationary subshift, for a Parry measure sequence (which is not even uniquely defined) we can only show a similar result under additional conditions on the sequence $(L_n)_{n=0}^{\infty}$. For instance, a simple sufficient condition implying the equality of metric and topological entropy (and hence the validity of the variational principle) is that the sequence $(L_n)$ is uniformly primitive, i.e., there is an integer $m$, such that for each $i$ all entries of $L_iL_{i+1}\cdots L_{i+m}$ are strictly positive.%

The paper is organized as follows: In Section \ref{sec_prelim} we introduce notation, give the definitions of the main quantities, and recall some of their properties. Section \ref{sec_suff_fam} contains the results on general topological nonautonomous systems, in particular the results on fine-scale IMS. The main results are Theorem \ref{thm_mecomp} about the computation of the metric entropy on admissible sequences of arbitrary small diameters and Theorem \ref{thm_constcond} about sufficient conditions for the existence of such sequences. Section \ref{sec_nsfts} contains the results about nonstationary subshifts of finite type, in particular Theorems \ref{thm_te_formula}, \ref{thm_nsft_me1} and \ref{thm_nsft_me2} about their entropy.%

\section{Preliminaries}\label{sec_prelim}%

A \emph{nonautonomous dynamical system} (an NDS, for short) is given by a pair $(X_{\infty},f_{\infty})$, where $X_{\infty} = (X_n)_{n=0}^{\infty}$ is a sequence of sets and $f_{\infty} = (f_n)_{n=0}^{\infty}$ a sequence of maps $f_n:X_n \rightarrow X_{n+1}$. For all $k\geq0$ and $n\geq1$ we put%
\begin{equation*}
  f_k^n := f_{k+n-1}  \circ \cdots \circ f_{k+1} \circ f_k,\quad f_k^0 := \id_{X_k}.%
\end{equation*}
Moreover, we define $f_k^{-n} := (f_k^n)^{-1}$, which is only applied to sets. (We do not assume that the maps $f_n$ are invertible.) The \emph{trajectory} of a point $x\in X_0$ is the sequence $(f_0^n(x))_{n=0}^{\infty}$. By $(X_{n,\infty},f_{n,\infty})$ we denote the pair of shifted sequences $X_{n,\infty} = (X_{n+k})_{k=0}^{\infty}$, $f_{n,\infty} = (f_{n+k})_{k=0}^{\infty}$ and we use similar notation for other sequences associated with an NDS. If such a sequence is constant, we drop the subscript ``$\infty$'', e.g., we write $X$ instead of $X_{\infty}$ if $X_n \equiv X$. If each $X_n$ is a compact metric space with associated metric $d_n$ and the sequence $f_{\infty}$ is equicontinuous, meaning that for every $\ep>0$ there is $\delta>0$ such that the implication $d_n(x,y) < \delta$ $\Rightarrow$ $d_{n+1}(f_n(x),f_n(y)) < \ep$ holds for all $x,y\in X_n$ and $n\geq0$, we speak of a \emph{topological NDS}. The assumption of equicontinuity rather than mere continuity of each $f_n$ is usually not part of the definition of a topological NDS (see \cite{Fis,KSn}). However, since it is essential for most results about topological entropy, including the power rule and the relation to metric entropy, we include it in the definition.%

For a topological NDS, the \emph{topological entropy} $h_{\tp}(f_{\infty})$ is defined as follows. For each sequence $\UC_{\infty} = (\UC_n)_{n=0}^{\infty}$ of open covers $\UC_n$ of $X_n$ we put%
\begin{equation*}
  h_{\tp}(f_{\infty};\UC_{\infty}) := \limsup_{n\rightarrow\infty}\frac{1}{n}\log\NC\Big(\bigvee_{i=0}^{n-1}f_0^{-i}\UC_i\Big),%
\end{equation*}
where $\bigvee$ is the usual join operation for open covers and $\NC(\cdot)$ stands for the minimal number of elements in a finite subcover. We write $\LC(X_{\infty})$ for the family of all sequences $\UC_{\infty}$ of open covers whose Lebesgue numbers are bounded away from zero and define the topological entropy of $f_{\infty}$ by%
\begin{equation}\label{eq_deftopent}
  h_{\tp}(f_{\infty}) := \sup_{\UC_{\infty}\in\LC(X_{\infty})}h_{\tp}(f_{\infty};\UC_{\infty}).%
\end{equation}
Unlike in the autonomous case, this quantity depends on the metrics $d_n$, and not only on the topologies of the spaces $X_n$. However, in the autonomous case, \eqref{eq_deftopent} gives the classical notion of topological entropy, which can be seen from the fact that $h_{\tp}(f_{\infty})$ can also be defined via $(n,\ep)$-separated or $(n,\ep)$-spanning sets as follows. We introduce the \emph{Bowen-metrics}%
\begin{equation*}
  d_{k,n}(x,y) := \max_{0 \leq i \leq n}d_{k+i}(f_k^i(x),f_k^i(y)),\quad k\geq0,\ n\geq0,%
\end{equation*}
on each of the spaces $X_k$. A subset $E \subset X_k$ is called \emph{$(n,\ep;f_{k,\infty})$-separated} if $d_{k,n}(x,y) \geq \ep$ holds for each two distinct points $x,y\in E$. A set $F \subset X_k$ \emph{$(n,\ep;f_{k,\infty})$-spans} another set $K \subset X_k$ if for each $x\in K$ there is $y\in F$ with $d_{k,n}(x,y) < \ep$. In the case $k=0$, we also speak of $(n,\ep)$-separated and $(n,\ep)$-spanning sets. We write $r_{\sep}(n,\ep;f_{k,\infty})$ for the maximal cardinality of an $(n,\ep;f_{k,\infty})$-separated set and $r_{\spn}(n,\ep;f_{k,\infty})$ for the minimal cardinality of set which $(n,\ep;f_{k,\infty})$-spans $X_k$. For each $\ep>0$ we define the numbers%
\begin{align*}
  h_{\sep}(\ep,f_{\infty}) &:= \limsup_{n\rightarrow\infty}\frac{1}{n}\log r_{\sep}(n,\ep;f_{\infty}),\\
	h_{\spn}(\ep,f_{\infty}) &:= \limsup_{n\rightarrow\infty}\frac{1}{n}\log r_{\spn}(n,\ep;f_{\infty}).%
\end{align*}
As in the autonomous case, it can be shown that (cf.~\cite{KSn,KMS})%
\begin{equation*}
  h_{\tp}(f_{\infty}) = \lim_{\ep\searrow0}h_{\sep}(\ep,f_{\infty}) = \lim_{\ep\searrow0}h_{\spn}(\ep,f_{\infty}),%
\end{equation*}
where the limits can be replaced by $\sup_{\ep>0}$.%

Let $\mu_{\infty} = (\mu_n)_{n=0}^{\infty}$ be a sequence of Borel probability measures for the metric spaces $X_n$ such that $f_n\mu_n \equiv \mu_{n+1}$, where $f_n$ here denotes the push-forward operator on measures, defined by $(f_n\mu)(A) := \mu(f_n^{-1}(A))$. We call such $\mu_{\infty}$ an \emph{invariant measure sequence (IMS)} for $(X_{\infty},f_{\infty})$ and associate to $(X_{\infty},f_{\infty},\mu_{\infty})$ a family $\EC_{\Mis} = \EC_{\Mis}(\mu_{\infty})$ of sequences of finite Borel partitions as follows. A sequence $\PC_{\infty} = (\PC_n)_{n=0}^{\infty}$, where $\PC_n = \{P_{n,1},\ldots,P_{n,k_n}\}$ is a Borel partition of $X_n$, belongs to $\EC_{\Mis}$ iff the sequence $(k_n)_{n=0}^{\infty}$ is bounded and for each $\ep>0$ there are $\delta>0$ and compact $K_{n,i} \subset P_{n,i}$ such that for all $n\geq0$ the following hold:%
\begin{enumerate}
\item[(a)] $\mu_n(P_{n,i} \backslash K_{n,i}) \leq \ep$ for $i = 1,\ldots,k_n$,%
\item[(b)] for all $1 \leq i < j \leq k_n$,%
\begin{equation*}
  D_n(K_{n,i},K_{n,j}) := \min_{(x,y) \in K_{n,i} \tm K_{n,j}} d_n(x,y) \geq \delta.%
\end{equation*}
\end{enumerate}
We call the sequences $\PC_{\infty}\in\EC_{\Mis}$ \emph{admissible}. The \emph{metric entropy of $f_{\infty}$ w.r.t.~$\PC_{\infty}$} for an arbitrary sequence $\PC_{\infty}$ of finite Borel partitions $\PC_n$ is defined as%
\begin{equation*}
  h(f_{\infty},\mu_{\infty};\PC_{\infty}) := \limsup_{n\rightarrow\infty}\frac{1}{n}H_{\mu_0}\Big(\bigvee_{i=0}^{n-1}f_0^{-i}\PC_i\Big),%
\end{equation*}
where $H_{\mu_0}(\PC) = - \sum_{P\in\PC}\mu_0(P)\log\mu_0(P)$ is the usual entropy of a partition. If the IMS $\mu_{\infty}$ is clear from the context, we will often omit this argument. The \emph{metric entropy of $f_{\infty}$ w.r.t.~$\mu_{\infty}$} is given by%
\begin{equation}\label{eq_defme}
  h(f_{\infty},\mu_{\infty}) := \sup_{\PC_{\infty}\in\EC_{\Mis}}h(f_{\infty},\mu_{\infty};\PC_{\infty}).%
\end{equation}
In \cite[Thm.~28]{Ka1} the inequality%
\begin{equation}\label{eq_varineq}
  h(f_{\infty},\mu_{\infty}) \leq h_{\tp}(f_{\infty})%
\end{equation}
was proved. If the equality%
\begin{equation}\label{eq_fvp}
  h_{\tp}(f_{\infty}) = \sup_{\mu_{\infty}}h(f_{\infty},\mu_{\infty})%
\end{equation}
is satisfied, the supremum taken over all IMS, we say that $(X_{\infty},f_{\infty})$ satisfies a \emph{full variational principle}.%

Two topological NDS $(X_{\infty},f_{\infty})$ and $(Y_{\infty},g_{\infty})$ are called \emph{equi-semiconjugate} if there exists an equicontinuous sequence $\pi_{\infty} = (\pi_n)_{n=0}^{\infty}$ of surjective maps $\pi_n:X_n \rightarrow Y_n$ such that $\pi_{n+1} \circ f_n = g_n \circ \pi_n$ for all $n\ge0$. The sequence $\pi_{\infty}$ is then called an \emph{equi-semiconjugacy}. In this case, $h_{\tp}(g_{\infty}) \leq h_{\tp}(f_{\infty})$. If the sequence $\pi_{\infty}$ consists of homeomorphisms and also $(\pi_n^{-1})_{n=0}^{\infty}$ is equicontinuous, then $\pi_{\infty}$ is called an \emph{equi-conjugacy} and $h_{\tp}(f_{\infty}) = h_{\tp}(g_{\infty})$.%

We end this section with the following simple observation:%

\begin{proposition}\label{prop_te_timeindep}
Let $(X_{\infty},f_{\infty})$ be a topological NDS such that each $f_n$ is surjective. Then $h_{\tp}(f_{k,\infty}) = h_{\tp}(f_{l,\infty})$ for all $k,l \geq 0$.%
\end{proposition}

\begin{proof}
By \cite[Lem.~4.5]{KSn} we have $h_{\tp}(f_{k,\infty}) \leq h_{\tp}(f_{k+1,\infty})$ for all $k\geq0$. The other inequality follows from the fact that the sequence $f_{k,\infty}$ is an equi-semiconjugacy from $f_{k,\infty}$ to $f_{k+1,\infty}$, i.e., from the commutative diagram%
\begin{eqnarray*}
\UseTips
\newdir{ >}{!/-5pt/\dir{>}}
  \xymatrix @=1pc @*[c] {
 X_k
   \ar[rr]^{f_k}
   \ar[dd]_{f_k}
 & & X_{k+1}
   \ar[rr]^{f_{k+1}}
   \ar[dd]_{f_{k+1}}
 & & X_{k+2}
   \ar[rr]^{f_{k+2}}
   \ar[dd]_{f_{k+2}}
 & & \ldots \\ \\
  X_{k+1}
   \ar[rr]_{f_{k+1}}
& & X_{k+2}
    \ar[rr]_{f_{k+2}}
& & X_{k+3}
    \ar[rr]_{f_{k+3}}
& & \ldots
  }
\end{eqnarray*}
This implies $h_{\tp}(f_{k+1,\infty}) \leq h_{\tp}(f_{k,\infty})$ for all $k\geq0$.
\end{proof}

\section{Fine-scale sequences}\label{sec_suff_fam}

In this section, we first prove a result which in many cases reduces the computation of the metric entropy $h(f_{\infty},\mu_{\infty})$ to the computation of $h(f_{\infty},\mu_{\infty};\PC_{\infty})$ for countably many $\PC_{\infty}$. The proof of this result is based on the inequality%
\begin{equation*}
  h(f_{\infty};\PC_{\infty}) \leq h(f_{\infty};\QC_{\infty}) + \limsup_{n\rightarrow\infty}\frac{1}{n}\sum_{i=0}^{n-1}H_{\mu_i}(\PC_i|\QC_i),%
\end{equation*}
that was shown in \cite[Prop.~9(vi)]{Ka1} and will be used to approximate $h(f_{\infty};\PC_{\infty})$ for arbitrary $\PC_{\infty} \in \EC_{\Mis}$ by the entropy on the elements $\QC_{\infty}$ of a countable subset of $\EC_{\Mis}$. To estimate the conditional entropy $H_{\mu_i}(\PC_i|\QC_i)$ we will use a result saying that the Rokhlin metric on the set of all partitions (of a probability space) of fixed cardinality $m$ is equivalent to another metric defined via the measures of symmetric set differences. For us, the crucial point of this result is that this equivalence is uniform with respect to the probability space (at least in one direction). This is made precise in the following lemma whose proof can be found in \cite[Prop.~4.3.5]{KHa}.%

\begin{lemma}\label{lem_metriceq}
For every integer $m\geq 2$ and $\ep>0$ there exists $\delta>0$ such that the following holds: Let $(X,\AC,\mu)$ be a probability space and let $\PC = \{P_1,\ldots,P_m\}$, $\QC = \{Q_1,\ldots,Q_m\}$ be two measurable partitions of $X$. Then the implication%
\begin{equation*}
  \inf_{\sigma}\sum_{i=1}^m \mu(P_i \triangle Q_{\sigma(i)}) < \delta \quad\Rightarrow\quad H_{\mu}(\PC|\QC) + H_{\mu}(\QC|\PC) < \ep%
\end{equation*}
holds, where the infimum is taken over all permutations $\sigma$ of $\{1,\ldots,m\}$.%
\end{lemma}

\begin{definition}
If $(X_n,d_n)_{n=0}^{\infty}$ is a sequence of compact metric spaces and $\PC_{\infty} = (\PC_n)_{n=0}^{\infty}$ an associated sequence of partitions, we call%
\begin{equation*}
  \diam\PC_{\infty} := \sup_{n\geq0}\sup_{P\in\PC_n}\diam P,%
\end{equation*}
the \emph{diameter} of $\PC_{\infty}$, where $\diam P = \sup_{x,y\in P}d_n(x,y)$.%
\end{definition}

\begin{theorem}\label{thm_mecomp}
Let $(X_{\infty},f_{\infty})$ be a topological NDS with an IMS $\mu_{\infty}$. Assume that there exists a sequence $(\RC_{\infty}^k)_{k=0}^{\infty}$ with $\RC_{\infty}^k \in \EC_{\Mis}(\mu_{\infty})$ such that%
\begin{equation*}
  \lim_{k\rightarrow\infty}\diam\RC_{\infty}^k = 0.%
\end{equation*}
Then the metric entropy satisfies%
\begin{equation*}
  h(f_{\infty},\mu_{\infty}) = \sup_{k\geq0}h(f_{\infty},\mu_{\infty};\RC_{\infty}^k) = \lim_{k\rightarrow\infty}h(f_{\infty},\mu_{\infty};\RC_{\infty}^k).%
\end{equation*}
\end{theorem}

\begin{proof}
Since we assume $\RC_{\infty}^k \in \EC_{\Mis}(\mu_{\infty})$, we clearly have%
\begin{equation*}
  h(f_{\infty},\mu_{\infty}) \geq \sup_{k\geq0}h(f_{\infty},\mu_{\infty};\RC_{\infty}^k).%
\end{equation*}
To prove the converse inequality, consider an arbitrary $\PC_{\infty} \in \EC_{\Mis}(\mu_{\infty})$. By adding sets of measure zero, we may assume that each $\PC_n$ has the same number $m$ of elements, say $\PC_n = \{P_{n,1},\ldots,P_{n,m}\}$. For a given $\ep>0$ we choose $\delta = \delta(m+1,\ep)$ according to Lemma \ref{lem_metriceq}. Since $\PC_{\infty}\in\EC_{\Mis}$, there exist $\rho = \rho(\delta) > 0$ and compact sets $K_{n,i}\subset P_{n,i}$ such that $\mu_n(P_{n,i}\backslash K_{n,i}) \leq \delta/m$ and $D_n(K_{n,i},K_{n,j}) \geq \rho$ for $i\neq j$. Choose $k$ large enough so that $\diam\RC^k_{\infty} < \rho/2$. For each of the sets $K_{n,i}$ consider the union of all elements of $\RC^k_n$ that have nonempty intersection with $K_{n,i}$ and denote this union by $Q_{n,i}$. This implies $Q_{n,i} \cap Q_{n,j} = \emptyset$ for $i\neq j$. Together with the complement $Q_{n,m+1} := X_n \backslash \bigcup_{i=1}^mQ_{n,i}$, we obtain a new partition%
\begin{equation*}
  \QC_n = \{ Q_{n,1},\ldots,Q_{n,m+1} \},%
\end{equation*}
and we consider the sequence $\QC_{\infty} := (\QC_n)_{n=0}^{\infty}$. Now we compare $h(f_{\infty};\QC_{\infty})$ with $h(f_{\infty};\RC^k_{\infty})$ and $h(f_{\infty};\PC_{\infty})$ with $h(f_{\infty};\QC_{\infty})$. Since $\RC^k_n$ is a refinement of $\QC_n$, by \cite[Prop.~9(iii)]{Ka1} we have%
\begin{equation*}
	h(f_{\infty};\QC_{\infty}) \leq h(f_{\infty};\RC^k_{\infty}).%
\end{equation*}
Using \cite[Prop.~9(vi)]{Ka1}, we get%
\begin{equation*}
  h(f_{\infty};\QC_{\infty}) \geq h(f_{\infty};\PC_{\infty}) - \limsup_{n\rightarrow\infty}\frac{1}{n}\sum_{i=0}^{n-1}H_{\mu_i}(\PC_i|\QC_i).%
\end{equation*}
Adding the empty set to $\PC_n$, we obtain a partition%
\begin{equation*}
  \widehat{\PC}_n = \{P_{n,1},\ldots,P_{n,m},P_{n,m+1}\},\quad P_{n,m+1}=\emptyset,%
\end{equation*}
with $H_{\mu_n}(\PC_n|\QC_n) = H_{\mu_n}(\widehat{\PC}_n|\QC_n)$. We have%
\begin{equation*}
  \mu_n(P_{n,i} \backslash Q_{n,i}) \leq \mu_n(P_{n,i}\backslash K_{n,i}) \leq \frac{\delta}{m}%
\end{equation*}
for $i=1,\ldots,m$ and%
\begin{equation*}
  \mu_n(Q_{n,i} \backslash P_{n,i}) \leq \sum_{j\neq i}\mu_n(Q_{n,i}\cap P_{n,j}) \leq \frac{m-1}{m}\delta,%
\end{equation*}
since $Q_{n,i}$ is disjoint from $K_{n,i}$. Moreover, $\mu_n(P_{n,m+1}\backslash Q_{n,m+1}) = 0$ and $\mu_n(Q_{n,m+1} \backslash P_{n,m+1}) = \mu_n(Q_{n,m+1}) \leq \delta$. Hence, the choice of $\delta$ yields $H_{\mu_n}(\QC_n|\PC_n) < \ep$, implying $h(f_{\infty};\RC^k_{\infty}) \geq h(f_{\infty};\PC_{\infty}) - \ep$. Since $\ep$ was chosen arbitrarily, this proves the inequality $h(f_{\infty},\mu_{\infty}) \leq \sup_{k\geq0} h(f_\infty,\mu_\infty;\RC_\infty^k)$. To see that the supremum over $k$ is equal to the limit for $k\rightarrow\infty$, observe that for each $\ep>0$ there is $\PC_{\infty}\in\EC_{\Mis}$ with $h(f_{\infty},\mu_{\infty};\PC_{\infty}) \geq h(f_{\infty},\mu_{\infty}) - \ep/2$ and there is $k_0$ such that $h(f_{\infty},\mu_{\infty};\RC^k_{\infty}) \geq h(f_{\infty},\mu_{\infty};\PC_{\infty}) - \ep/2$ for all $k\geq k_0$. Hence, $h(f_{\infty},\mu_{\infty}) \geq h(f_{\infty};\RC^k_{\infty}) \geq h(f_{\infty},\mu_{\infty}) - \ep$  for all $k \geq k_0$.%
\end{proof}

For many arguments in the classical entropy theory it is essential that one can use partitions of arbitrarily small diameter. To establish a reasonable entropy theory for NDS, it hence is important to understand under which conditions there exist admissible sequences with arbitrarily small diameters as required in the preceding theorem. This motivates the following definition.%

\begin{definition}
Let $(X_{\infty},f_{\infty})$ be a topological NDS with an IMS $\mu_{\infty}$. We call a sequence $(\PC^k_{\infty})_{k=0}^{\infty}$ in $\EC_{\Mis}(\mu_{\infty})$ a \emph{fine-scale sequence} if $\diam\PC^k_{\infty}\rightarrow0$. The IMS $\mu_{\infty}$ is called a \emph{fine-scale IMS} if there exists a fine-scale sequence in $\EC_{\Mis}(\mu_{\infty})$.%
\end{definition}

The next proposition summarizes elementary properties of systems with fine-scale IMS, see also Proposition \ref{prop_productext}.%

\begin{proposition}\label{prop_finescaleprops}
The following assertions hold:%
\begin{enumerate}
\item[(i)] A topological NDS $(X_{\infty},f_{\infty})$ has a fine-scale IMS iff the sequence $X_{\infty}$ is uniformly totally bounded, i.e., for each $\alpha>0$ there is $m\in\N$ such that $m$ balls of radius $\alpha$ are sufficient to cover $X_n$ for any $n\geq0$.%
\item[(ii)] Let $(X_{\infty},f_{\infty})$ and $(Y_{\infty},g_{\infty})$ be topological NDS that are equi-conjugate via $\pi_{\infty} = (\pi_n)_{n=0}^{\infty}$. Then, if $\mu_{\infty}$ is a fine-scale IMS for $(X_{\infty},f_{\infty})$, the sequence $\nu_{\infty} = (\nu_n)_{n=0}^{\infty}$ given by $\nu_n = \pi_n\mu_n$ is a fine-scale IMS for $(Y_{\infty},g_{\infty})$.%
\item[(iii)] Let $(X_{\infty},f_{\infty})$ and $(Y_{\infty},g_{\infty})$ be topological NDS with fine-scale IMS $\mu_{\infty}$ and $\nu_{\infty}$, respectively. Then $\mu_{\infty} \tm \nu_{\infty}$, defined componentwise by $\mu_n \tm \nu_n$, is a fine-scale IMS for the direct product system $(X_{\infty} \tm Y_{\infty},f_{\infty}\tm g_{\infty})$ (also defined componentwise) and%
\begin{equation}\label{eq_productineq}
  h_{\EC_{\Mis}(\mu_{\infty}\tm\nu_{\infty})}(f_{\infty}\tm g_{\infty}) \leq h_{\EC_{\Mis}(\mu_{\infty})}(f_{\infty}) + h_{\EC_{\Mis}(\nu_{\infty})}(g_{\infty}),%
\end{equation}
where on $X_n\tm Y_n$ we use the product metrics%
\begin{equation*}
  d_n^{\tm}((x_1,y_1),(x_2,y_2)) = \max\{d_n^X(x_1,x_2),d_n^Y(y_1,y_2)\}.%
\end{equation*}
\end{enumerate}
\end{proposition}

\begin{proof}
To prove (i), assume that $\mu_{\infty}$ is a fine-scale IMS for $(X_{\infty},f_{\infty})$. Then for each $\alpha>0$ we can find an admissible sequence $\PC_{\infty}$ with $\diam\PC_{\infty} < \alpha$. The fact that $\PC_{\infty}$ is admissible in particular implies $\#\PC_n \leq m$ for some $m\in\N$. Since each $P\in\PC_n$ is contained in the $\alpha$-ball around any $x\in P$, we see that $m$ $\alpha$-balls are sufficient to cover $X_n$, and hence $X_{\infty}$ is uniformly totally bounded. Conversely, assume that $X_{\infty}$ is uniformly totally bounded. Let $\mu_0 := \delta_{x_0}$ for an arbitrary $x_0 \in X_0$ and consider the IMS $\mu_n = f_0^n\mu_0 = \delta_{f_0^n(x_0)}$. For any given $\alpha>0$ choose $m$ so that $m$ balls of radius $\alpha/2$ are sufficient to cover $X_n$ for each $n$. From such a ball-cover of $X_n$ one easily constructs a measurable partition $\PC_n = \{P_{n,1},\ldots,P_{n,m}\}$ with $m$ elements such that $\diam\PC_n < \alpha$, by cutting away the overlaps between the balls. It is easily seen that such a sequence $\PC_{\infty} = (\PC_n)_{n=0}^{\infty}$ is admissible. Indeed, for any $\ep>0$ let $K_{n,i} \subset P_{n,i}$ be defined by $K_{n,i} := \{f_0^n(x_0)\}$ if $f_0^n(x_0) \in P_{n,i}$ and $K_{n,i} := \emptyset$ otherwise. Then $K_{n,i}$ is a compact subset of $P_{n,i}$ and the conditions that $\mu_n(P_{n,i}\backslash K_{n,i}) \leq \ep$ and $d_n(x,y) \geq \delta > 0$ and $x\in K_{n,i}$, $y\in K_{n,j}$ ($i\neq j$) are trivially satisfied.%

For (ii) we note that in \cite[Prop.~27]{Ka1} it was proved that $\nu_{\infty}$ is an IMS for $(Y_{\infty},g_{\infty})$ and $\EC_{\Mis}(\mu_{\infty})$ and $\EC_{\Mis}(\nu_{\infty})$ are isomorphic in the sense that $\PC_{\infty} = (\PC_n)_{n=0}^{\infty} \in \EC_{\Mis}(\mu_{\infty})$ iff $(\pi_n\PC_n)_{n=0}^{\infty} \in \EC_{\Mis}(\nu_{\infty})$. From the equicontinuity of $\pi_{\infty}$ it easily follows that a fine-scale sequence in $\EC_{\Mis}(\mu_{\infty})$ yields a fine-scale sequence in $\EC_{\Mis}(\nu_{\infty})$ via this isomorphism.%

Finally, let us show (iii). It follows from a simple computation that the sequence $(\mu_n \tm \nu_n)_{n=0}^{\infty}$ of Borel probability measures on the spaces $X_n\tm Y_n$ is an IMS for the product system $(X_{\infty} \tm Y_{\infty},f_{\infty}\tm g_{\infty})$. By the assumption, we can choose for any given $\alpha>0$ sequences $\PC_{\infty}\in\EC_{\Mis}(\mu_{\infty})$ and $\QC_{\infty}\in\EC_{\Mis}(\nu_{\infty})$ such that $\diam\PC_{\infty},\diam\QC_{\infty} \leq \alpha$. Consider the sequence $\PC_{\infty}\tm\QC_{\infty}$ of product partitions%
\begin{equation*}
  \PC_n \tm \QC_n := \left\{ P \tm Q\ :\ P \in \PC_n,\ Q \in \QC_n \right\},\quad n \geq 0.%
\end{equation*}
For each $P \tm Q \in \PC_n \tm \QC_n$ we find $\diam(P\tm Q) \leq \alpha$ in the product metric $d_n^{\tm}$, hence $\diam(\PC_{\infty}\tm\QC_{\infty}) \leq \alpha$.  We claim that $\PC_{\infty} \tm \QC_{\infty}$ is in $\EC_{\Mis}(\mu_{\infty}\tm\nu_{\infty})$. To show this, assume $\PC_n = \{P_{n,1},\ldots,P_{n,k_n}\}$ and $\QC_n = \{Q_{n,1},\ldots,Q_{n,l_n}\}$. For a given $\ep>0$ let $\delta_1 = \delta(\PC_{\infty},\ep/2) > 0$ and $\delta_2 = \delta(\QC_{\infty},\ep/2) > 0$ as well as compact sets $K_{n,i} \subset P_{n,i}$ and $L_{n,i} \subset Q_{n,i}$ be chosen according to the definition of $\EC_{\Mis}$. Then%
\begin{eqnarray*}
 && \mu_n \tm \nu_n\left((P_{n,i}\tm Q_{n,i})\backslash (K_{n,i}\tm L_{n,i})\right)\\
&& \qquad = \mu_n \tm \nu_n\left((P_{n,i}\backslash K_{n,i}) \tm Q_{n,i} \cup P_{n,i} \tm (Q_{n,i}\backslash L_{n,i})\right)\\
&& \qquad \leq \mu_n(P_{n,i}\backslash K_{n,i})\nu_n(Q_{n,i}) + \mu_n(P_{n,i})\nu_n(Q_{n,i}\backslash L_{n,i})
\leq \frac{\ep}{2} + \frac{\ep}{2} = \ep.%
\end{eqnarray*}
Moreover, for $(x_1,y_1) \in K_{n,i_1} \tm L_{n,j_1}$ and $(x_2,y_2) \in K_{n,i_2} \tm L_{n,j_2}$ with $(i_1,j_1) \neq (i_2,j_2)$ we have%
\begin{equation*}
  d_n^{\tm}((x_1,y_1),(x_2,y_2)) = \max\left\{d^X_n(x_1,x_2),d_n^Y(y_1,y_2)\right\} \geq \min\{\delta_1,\delta_2\} =: \delta > 0.%
\end{equation*}
This proves the claim. Consequently, since $\alpha$ was chosen arbitrarily, $\mu_{\infty} \tm \nu_{\infty}$ is a fine-scale IMS. It remains to prove the entropy inequality. First note that for $\PC_{\infty}\in\EC_{\Mis}(\mu_{\infty})$ and $\QC_{\infty}\in\EC_{\Mis}(\nu_{\infty})$ we have%
\begin{align*}
  h(f_{\infty} \tm g_{\infty};\PC_{\infty} \tm \QC_{\infty}) &= \limsup_{n\rightarrow\infty}\frac{1}{n}H_{\mu_0\tm\nu_0}\Big(\bigvee_{i=0}^{n-1}(f_0^i \tm g_0^i)^{-1}(\PC_i\tm\QC_i)\Big)\\
	&= \limsup_{n\rightarrow\infty}\frac{1}{n}H_{\mu_0\tm\nu_0}\Big(\bigvee_{i=0}^{n-1}f_0^{-i}\PC_i \tm \bigvee_{i=0}^{n-1}g_0^{-i}\QC_i\Big).%
\end{align*}
For any measures $\mu,\nu$ and partitions $\PC,\QC$ the identity $H_{\mu \tm \nu}(\PC \tm \QC) = H_{\mu}(\PC) + H_{\nu}(\QC)$ holds, as can easily be seen. This gives%
\begin{equation*}
  h(f_{\infty} \tm g_{\infty};\PC_{\infty} \tm \QC_{\infty}) \leq h(f_{\infty};\PC_{\infty}) + h(g_{\infty};\QC_{\infty}).%
\end{equation*}
The inequality \eqref{eq_productineq} now follows by considering $\PC_{\infty}$ and $\QC_{\infty}$ of arbitrarily small diameter and applying Theorem \ref{thm_mecomp} to the product system.%
\end{proof}

\begin{remark}
In general, we cannot expect equality in \eqref{eq_productineq}, since the corresponding $\limsup$'s for sequences $\PC_{\infty}\tm\QC_{\infty}$, $\PC_{\infty}$ and $\QC_{\infty}$ may be attained on different subsequences. See also Hulse \cite{Hul} for counterexamples to the product formula for sequence entropy and topological entropy on non-compact spaces that could be adapted to produce a counterexample for the equality in \eqref{eq_productineq}.%
\end{remark}

To find sufficient conditions for determining whether a given IMS is a fine-scale IMS, we first consider systems with stationary state space $X_n\equiv X$. For the stationary case, in Theorem \ref{thm_constcond} and Corollary \ref{cor_finescaleims} below we give sufficient conditions for the existence of fine-scale sequences $(\PC^k_{\infty})_{k=0}^{\infty}$ with each $\PC^k_{\infty}$ being a constant sequence, while in Proposition \ref{prop_3.16} we show an example for which such sequences do not exist.%

\begin{definition}
We say that a topological NDS $(X_{\infty},f_{\infty})$ has a \emph{stationary state space} if all $(X_n,d_n)$ are identical, i.e., $(X_n,d_n) = (X,d)$. In this case, we simply write $(X,f_{\infty})$ instead of $(X_{\infty},f_{\infty})$. Given an IMS $\mu_{\infty}$ of $(X,f_{\infty})$, we write $\WC(\mu_{\infty})$ for the set of all weak$^*$ limit points of $\mu_{\infty}$.%
\end{definition}

\begin{theorem}\label{thm_constcond}
Let $(X,f_{\infty})$ be a topological NDS with stationary state space and an IMS $\mu_{\infty}$. Then, under each of the following conditions, there exist fine-scale sequences $(\PC^k_{\infty})_{k=0}^{\infty}$ with constant $\PC^k_{\infty}$.%
\begin{enumerate}
\item[(i)] The set $\{\mu_n\}_{n=0}^{\infty}$ is relatively compact in the strong topology on the space of measures.%
\item[(ii)] For every $\alpha>0$ there is a finite measurable partition $\PC$ of $X$ with $\diam\PC<\alpha$ such that every $P\in\PC$ satisfies $\nu(\partial P)=0$ for all $\nu\in\WC(\mu_{\infty})$.%
\item[(iii)] The space $X$ is zero-dimensional.%
\item[(iv)] $X = [0,1]$ or $X = \rmS^1$ and there exists a dense set $D \subset X$ such that every $x\in D$ satisfies $\nu(\{x\}) = 0$ for all $\nu\in\WC(\mu_{\infty})$.%
\end{enumerate}
\end{theorem}

\begin{proof}
Under condition (i), it was proved in \cite[Prop.~31]{Ka1} that $\EC_{\Mis}$ contains all constant sequences.%

Assume (ii). For a given $\alpha>0$ let $\PC = \{P_1,\ldots,P_k\}$ be a partition with $\diam\PC<\alpha$ and $\nu(\partial\PC) = 0$ for all $\nu \in \WC(\mu_{\infty})$. We claim that the constant sequence $\PC_n \equiv \PC$ is admissible. To prove this, let $\ep>0$. For every $\nu \in \WC(\mu_{\infty})$ and every $i\in\{1,\ldots,k\}$ we find a compact set $K_i^{\nu} \subset P_i$ such that $\nu(P_i \backslash K_i^{\nu}) \leq \ep/2$ and $\nu(\partial(P_i \backslash K_i^{\nu})) = 0$ (see \cite[Lem.~33]{Ka1} for the existence of such $K_i^{\nu}$). By the Portmanteau theorem, there exists a weak$^*$-neighborhood $U_{\nu}$ of $\nu$ such that every $\mu \in U_{\nu}$ satisfies $|\mu(P_i\backslash K_i^{\nu}) - \nu(P_i\backslash K_i^{\nu})| \leq \ep/2$ for $i=1,\ldots,k$. By compactness of $\WC(\mu_{\infty})$, finitely many such neighborhoods, say $U_1 = U_{\nu_1},\ldots,U_r = U_{\nu_r}$, are sufficient to cover $\WC(\mu_{\infty})$. By standard arguments, it follows that there exists $n_0\geq1$ such that $\mu_n \in \bigcup_{j=1}^r U_j$ for all $n\geq n_0$. For all $n \in \{0,\ldots,n_0-1\}$ we can find compact sets $K_i^n \subset P_i$ with $\mu_n(P_i \backslash K_i^n) \leq \ep$ for $i=1,\ldots,k$. Let%
\begin{equation*}
  K_i := \bigcup_{j=1}^r K_i^{\nu_j} \cup \bigcup_{n=0}^{n_0-1}K_i^n,\quad i=1,\ldots,k.%
\end{equation*}
Then $K_i$ is a compact subset of $P_i$. For every $n\geq n_0$ we find $j_n \in \{1,\ldots,r\}$ such that $\mu_n \in U_{\nu_{j_n}}$, implying%
\begin{equation*}
  \mu_n(P_i \backslash K_i) \leq \mu_n(P_i \backslash K_i^{\nu_{j_n}}) \leq \nu_{j_n}(P_i \backslash K_i^{\nu_{j_n}}) + \frac{\ep}{2} \leq \frac{\ep}{2} + \frac{\ep}{2} = \ep.%
\end{equation*}
For $n < n_0$ we have $\mu_n(P_i \backslash K_i) \leq \mu_n(P_i \backslash K_i^n) \leq \ep$, proving the claim.%

Now assume that condition (iii) holds. Then the topology of $X$ has a base consisting of clopen sets (in fact, this is one possible definition of a zero-dimensional space). Hence, we can find for each $\alpha>0$ a partition $\PC = \{P_1,\ldots,P_k\}$ of $X$ with compact sets $P_i$ satisfying $\diam P_i < \alpha$ for $i=1,\ldots,k$. Let $\delta := \min_{i\neq j}D(P_i,P_j)$. By compactness we have $\delta>0$, which obviously implies that $\PC_n \equiv \PC$ is contained in $\EC_{\Mis}$.%

Finally, suppose that (iv) holds. We only give the proof for $X= [0,1]$, since for $\rmS^1$ it is very similar. We first prove that for every $x\in D$ the following holds:%
\begin{equation*}
  \forall \ep>0\ \exists \delta>0:\ \forall n\geq 0,\ \mu_n(B_{\delta}(x)) < \ep.%
\end{equation*}
To show this, we argue by contradiction, i.e., we assume that there exists $\ep>0$ such that for every $\delta>0$ there is $n = n(\delta)$ with $\mu_n(B_{\delta}(x)) \geq \ep$. Let $(\delta_k)_{k=0}^{\infty}$ converge to zero and let $n_k = n(\delta_k)$, i.e.,%
\begin{equation*}
  \mu_{n_k}(B_{\delta_k}(x)) \geq \ep \mbox{\quad for all\ } k \geq 0.%
\end{equation*}
We may assume that $\mu_{n_k}$ weakly$^*$-converges to some $\mu$. Then we may replace the numbers $\delta_k$ by slightly bigger numbers so that the convergence to zero still holds, but now $\mu(\partial B_{\delta_k}(x)) = 0$. Fix any $k_* \geq 0$. Using the Portmanteau theorem and the fact that $B_{\delta_k}(x)\subset B_{\delta_{k_*}}(x)$ for sufficiently large $k$, we find%
\begin{equation*}
  \mu(B_{\delta_{k_*}}(x)) = \lim_{k\rightarrow\infty}\mu_{n_k}(B_{\delta_{k_*}}(x)) \geq \ep.%
\end{equation*}
Letting $k_* \rightarrow \infty$, we thus obtain the contradiction $\mu(\{x\}) \geq \ep$. Now, for a given $\alpha>0$, we partition the interval $[0,1]$ into subintervals $I_1,\ldots,I_k$ of uniform length $1/k$, where $k > 2/\alpha$ and $\sup I_j = \inf I_{j+1}$. By the claim, we can pick $x_j \in I_j$ for $j = 1,\ldots,k$ such that%
\begin{equation}\label{eq_epdelta}
  \forall \ep>0\ \exists \delta>0:\ \forall n\geq 0,\ j=1,\ldots,k,\ \mu_n(B_{\delta}(x_j)) < \ep.%
\end{equation}
Let the partition $\PC$ consist of the subintervals $P_1 := [0,x_1)$, $P_2 := [x_1,x_2)$, $\ldots$, $P_{k+1} := [x_k,1]$. Then $\diam\PC \leq 2/k < \alpha$. To show that the constant sequence $\PC_n \equiv \PC$ is admissible, let $\ep>0$ and pick $\delta = \delta(\ep/2)$ according to \eqref{eq_epdelta}. Define compact sets%
\begin{align*}
  K_j &:= \left\{ x\in P_j\ :\ \mathrm{dist}(x,\partial P_j) \geq \delta/2 \right\},\quad j=2,\ldots,k,\\
  K_1 &:= \left\{ x\in P_1\ :\ |x-x_1| \leq \delta/2 \right\},\ K_{k+1} := \left\{ x\in P_{k+1}\ :\ |x-x_k| \leq \delta/2 \right\}.%
\end{align*}
This implies that any two $K_j$ have distance $\geq\delta$ and $\mu_n(P_j \backslash K_j)\leq\ep$.%
\end{proof}

\begin{remark}
We note that condition (ii) in the preceding theorem is reminiscent of the \emph{small boundary property} in the entropy structure theory of autonomous dynamical systems (cf.~Downarowicz \cite{Dow,Do2}). Here partitions of a space $X$ are of interest with boundaries of measure zero w.r.t.~all invariant measures of a given map $T$ on $X$.%
\end{remark}

From Theorem \ref{thm_mecomp} combined with Theorem \ref{thm_constcond}(i) or (ii) we immediately obtain%

\begin{corollary}\label{cor_autonomous}
For an autonomous measure-preserving dynamical system, the metric entropy defined as in \eqref{eq_defme} is the same as the usual Kolmogorov-Sinai metric entropy.%
\end{corollary}

In the next corollary we formulate conditions that are easier to check than the conditions in Theorem \ref{thm_constcond}.%

\begin{corollary}\label{cor_finescaleims}
Let $(X,f_{\infty})$ be a topological NDS with stationary state space and an IMS $\mu_{\infty}$. Then, under each of the following conditions, there exist fine-scale sequences $(\PC^k_{\infty})_{k=0}^{\infty}$ with constant $\PC^k_{\infty}$.%
\begin{enumerate}
\item[(i)] $\WC(\mu_{\infty})$ contains at most countably many non-equivalent measures.%
\item[(ii)] $X = [0,1]$ or $X = \rmS^1$ and at most countably many $\mu\in\WC(\mu_{\infty})$ admit one-point sets of positive measure.%
\item[(iii)] The space $X$ is of the form $X = Y \tm I$ with a compact metric space $Y$ and $\mu_n = \nu_n \tm \lambda$, where $\lambda$ is the standard Lebesgue measure on $I=[0,1]$.%
\end{enumerate}
\end{corollary}

\begin{proof}
Under condition (i) it can be concluded from Theorem \ref{thm_constcond}(ii) that the assertion holds. Indeed, by \cite[Fact 6.6.6]{Do2}, for any countable family of probability measures of a compact metric space, there exist arbitrarily fine partitions with boundaries of measure zero w.r.t.~all measures in this family.%

Under condition (ii) the proof follows from Theorem \ref{thm_constcond}(iv): Let $\{\nu_{\alpha}\}_{\alpha\in A}$ be the set of all elements of $\WC(\mu_{\infty})$ that admit one-point sets of positive measure and let $A_{\alpha} := \{ x\in X : \nu_{\alpha}(\{x\}) > 0 \}$. Then $A_{\alpha}$ is countable, since otherwise one of the sets $\{x \in X : \nu_{\alpha}(\{x\}) > 1/k\}$, $k\in\N$, would be infinite in contradiction to the finiteness of $\nu_{\alpha}$. As a consequence, also $A := \bigcup_{\alpha\in A}A_{\alpha}$ is countable, and hence $D := X \backslash A$ is dense in $X$. For every $x\in D$ we have $\nu(\{x\}) = 0$ for each $\nu\in\WC(\mu_{\infty})$. This shows that condition (iv) in Theorem \ref{thm_constcond} is satisfied.%

Under condition (iii) the proof follows from Theorem \ref{thm_constcond}(ii). Indeed, we can construct partitions of $X$ with arbitrarily small diameters and zero boundaries with respect to \emph{all} probability measures of the form $\mu \tm \lambda$ on $Y \tm I$, so in particular for all elements of $\WC(\mu_{\infty}\tm\lambda)$, as follows. Take a continuous function $\varphi:Y\rightarrow I$ and define%
\begin{eqnarray*}
  P^a(\varphi) &:=& \left\{ (y,t) \in Y \tm I\ :\ t \geq \varphi(y) \right\},\\
	P^b(\varphi) &:=& \left\{ (y,t) \in Y \tm I\ :\ t < \varphi(y) \right\}.%
\end{eqnarray*}
This yields the measurable partition $\PC(\varphi) = \{P^a(\varphi),P^b(\varphi)\}$. For finitely many continuous functions $\varphi_1,\ldots,\varphi_m:Y\rightarrow I$ we also define the partition $\PC(\varphi_1,\ldots,\varphi_m) := \bigvee_{j=1}^m \PC(\varphi_j)$. We observe that $\partial P^a(\varphi) \cup \partial P^b(\varphi) \subset \Graph(\varphi)$. For any partition of $I$ into subintervals $I_1,I_2,\ldots,I_r$ of length $1/r$,%
\begin{equation*}
  \Graph(\varphi) \subset \bigcup_{j=1}^r \varphi^{-1}(I_j) \tm I_j,%
\end{equation*}
implying%
\begin{equation*}
  (\mu \tm \lambda)(\Graph(\varphi)) \leq \sum_{j=1}^r \mu(\varphi^{-1}(I_j))\lambda(I_j) = \frac{1}{r}.%
\end{equation*}
Hence, $(\mu \tm \lambda)(\partial P^a(\varphi)) = (\mu \tm \lambda)(\partial P^b(\varphi)) = 0$. Since the elements of $\PC(\varphi_1,\ldots,\varphi_m)$ are finite intersections of sets with zero boundaries w.r.t.~all $\mu\tm\lambda$, their boundaries also have this property. According to \cite[Sec.~6.2]{Dow} we can choose $(\varphi_j)_{j=1}^{\infty}$ so that $\diam\PC(\varphi_1,\ldots,\varphi_m) \rightarrow 0$ for $m\rightarrow\infty$.%
\end{proof}

\begin{remark}
Here is a possible application of condition (iii) above: Consider the identity map $\id:I \rightarrow I$ and note that $h_{\tp}(f_{\infty} \tm \id) = h_{\tp}(f_{\infty})$, which is easy to prove. However, in general we don't know if $h_{\EC_{\Mis}(\mu_{\infty}\tm\lambda)}(f_{\infty} \tm \id) = h_{\EC_{\Mis}(\mu_{\infty})}(f_{\infty})$. Proposition \ref{prop_productext} given below yields this equality only under the assumption that $\mu_{\infty}$ is a fine-scale IMS. Now condition (iii) shows that $\mu_{\infty} \tm \lambda$ is always a fine-scale IMS for $f_{\infty} \tm \id$. Hence, it might be easier to compute lower bounds for $h_{\tp}(f_{\infty})$ in terms of the metric entropy of the product system $f_{\infty} \tm \id$ than by using $f_{\infty}$. Condition (iii) together with Proposition \ref{prop_productext} also shows that for a fine-scale IMS one can always use constant sequences of partitions to compute the metric entropy by passing over to the product $(f_{\infty}\tm\id,\mu_{\infty}\tm\lambda)$.%
\end{remark}

With regard to condition (i) in the preceding corollary, we note that the number of non-equivalent elements of $\WC(\mu_{\infty})$ is \emph{not} an invariant with respect to equi-conjugacies, see Example \ref{exm_3.14} below. This is very interesting, because according to Proposition \ref{prop_finescaleprops}(ii), equi-conjugacies preserve the property of being a fine-scale IMS. Hence, we can formulate the following corollary.%

\begin{corollary}\label{cor_countablelimitset}
Let $(X_{\infty},f_{\infty})$ be a topological NDS with an IMS $\mu_{\infty}$. Let $Y$ be a compact metric space and $\pi_n:X_n \rightarrow Y$ homeomorphisms such that $(\pi_n)_{n=0}^{\infty}$ and $(\pi_n^{-1})_{n=0}^{\infty}$ are equicontinuous. If there are at most countably many non-equivalent measures in $\WC((\pi_n\mu_n)_{n=0}^{\infty})$, then $\mu_{\infty}$ is a fine-scale IMS.%
\end{corollary}

\begin{example}\label{exm_3.14}
Let $\rmS^1 \subset \C$ be the unit circle and consider a map $f:\rmS^1 \rightarrow \rmS^1$ with a dense forward orbit $\{f^n(x_0)\}_{n=0}^{\infty}$ (e.g., an irrational rotation or the angle doubling map). Let $f_n := f$ for every $n\geq0$ and put $\mu_0 := \delta_{x_0}$, $\mu_n := f^n\mu_0 = \delta_{f^n(x_0)}$. This IMS has every Dirac measure $\delta_x$, $x\in\rmS^1$, as a limit point. Choose for each $n$ a rotation $\pi_n:\rmS^1 \rightarrow \rmS^1$ such that $\pi_n(f^n(x_0)) = 1$. Then $\{\pi_n\}$ and $\{\pi_n^{-1}\}$ are equicontinuous and $\pi_n\mu_n$ equals $\delta_1$ constantly.%
\end{example}

We next present an example of an NDS on the unit interval, which does not admit constant admissible sequences of partitions into subintervals of arbitrarily small diameters. We will use the following lemma.%

\begin{lemma}\label{lem_hl}
Let $n\in\N$ and $J \subset (0,1)$ be a compact interval of length $2^{-n}$. Then there are $k=k(n)\in\N$ and continuous piecewise affine maps $f_1,\ldots,f_k:[0,1]\rightarrow[0,1]$ such that%
\begin{enumerate}
\item[(i)] For every $i\in\{1,\ldots,k\}$, the restriction
 $ f_i|_{f_1^{i-1}(J)}:f_1^{i-1}(J) \rightarrow f_1^i(J)$ is an isometry.%
\item[(ii)] For every $x$ with $2^{-(n+1)} < x < 1 - 2^{-(n+1)}$ there is $i \in \{1,\ldots,k\}$  with
$x \in f_1^i(J)$ and $\dist(x,\partial f_1^i(J)) > 2^{-(n+2)}$.%
\item[(iii)] $|f_i'(x)| \leq 1$ for $i=1,\ldots,k$ and all $x$ such that $f_i'(x)$ exists.%
\end{enumerate}
\end{lemma}

\begin{proof}
Consider the intervals%
\begin{equation*}
  J_i := \left[\frac{i}{2^{n+2}},\frac{i}{2^{n+2}} + \frac{1}{2^n}\right] = \left[\frac{i}{2^{n+2}},\frac{i+4}{2^{n+2}}\right],\quad 1 \leq i \leq 2^{n+2} - 5 =: k.%
\end{equation*}
Let $x\in [2^{-(n+1)},1-2^{-(n+1)}]$. Then $x$ is contained in an interval of the form%
\begin{equation*}
  \left[\frac{i}{2^{n+2}},\frac{i+1}{2^{n+2}}\right],\quad i \in \{2,\ldots,2^{n+2}-3\}.%
\end{equation*}
This implies $x \in J_{i-1}$. Moreover, $\dist(x,\partial J_{i-1}) \geq 2^{-(n+2)}$, because%
\begin{equation*}
  x - \frac{i-1}{2^{n+2}} \geq \frac{i}{2^{n+2}} - \frac{i-1}{2^{n+2}} = \frac{1}{2^{n+2}}%
\end{equation*}
and%
\begin{equation*}
  \frac{i-1}{2^{n+2}} + \frac{1}{2^n} - x \geq \frac{i-1}{2^{n+2}} + \frac{1}{2^n} - \frac{i+1}{2^{n+2}} = \frac{1}{2^{n+1}}.%
\end{equation*}
The maps $f_1,\ldots,f_k$ are constructed as follows: if $J = [a,b]$, then $f_1$ can be chosen as a piecewise affine map with $f_1(a) = 2^{-(n+2)}$, $f_1(b) = 2^{-(n+2)} + 2^{-n}$ such that $f_1$ has slope $1$ on $[a,b]$ and slope $\leq1$ on $[0,a]$ and $[b,1]$. The maps $f_i$ for $2 \leq i \leq k$ are chosen similarly, mapping $J_{i-1}$ to $J_i$ isometrically.%
\end{proof}

\begin{proposition}\label{prop_3.16}
There exists an NDS $(I,f_{\infty})$ on the unit interval $I$ with an IMS $\mu_{\infty}$ such that $\EC_{\Mis}(\mu_{\infty})$ does not contain any constant sequence of partitions into subintervals of diameter less than one.%
\end{proposition}

\begin{proof}
The proof proceeds in two steps.%

\emph{Step 1.} Informally, the idea of the construction is as follows: we start with the standard Lebesgue measure as an initial measure $\mu_0$ and choose the first map $f_0$ in the sequence so that it scales the Lebesgue measure down to an interval of length $1/2$. Then we apply finitely many maps that just translate the support of the scaled-down Lebesgue measure around the unit interval so that many points will be contained in the interior of the support with some uniform distance to the boundary. Then we scale down by factor $1/2$ again and move the support around in the same fashion, and so on. By this construction, the support of the measure becomes a smaller and smaller interval when $n$ becomes large, and is moved around all over the interval.%

Formally, we construct $f_{\infty}$ as follows. The map $f_0$ is an affine map of slope $1/2$ such that the interval $J_0 := f_0(I) \subset (0,1)$. Then $f_1,\ldots,f_{n_1}$ are piecewise affine maps chosen according to Lemma \ref{lem_hl} with $J = J_0$, $n=1$ and $n_1=k(1)$. Next, $f_{n_1 + 1}$ is chosen to be an affine map of slope $1/2$ that maps $f_0^{n_1+1}(I) = f_1^{n_1}(J_0)$ into $(0,1)$. Again, $f_{n_1+2},\ldots,f_{n_2}$ are piecewise affine maps chosen according to Lemma \ref{lem_hl} with $J = f_1^{n_1}(J_0)$, $n=2$
and $n_2=n_1+k(2)+1$. Going on in this way, we produce a sequence $f_{\infty} = (f_n)_{n=0}^{\infty}$ that has the following property:%

For every $x \in (0,1)$ there is an $n_0$ such that for all $n\geq n_0$, $x$ is contained in a subinterval $J_n\subset I$ of length $2^{-n}$ with $\dist(x,\partial J_n) > 2^{-(n+2)}$, and each $J_n$ is of the form $J_n = f_0^{k_n}(I)$ for some $k_n$.%

\emph{Step 2.} We let $\mu_0$ be the standard Lebesgue measure on $I$ and consider the induced IMS $\mu_n = f_0^n\mu_0$. By construction, $\mu_n$ is the normalized Lebesgue measure on the interval $f_0^n(I)$. Now pick $x\in(0,1)$ and choose $n_0$ according to the result of Step 1. Then for all $n\geq n_0$ we obtain%
\begin{equation*}
  \mu_{k_n}([x-2^{-(n+2)},x]) = \mu_{k_n}([x,x+2^{-(n+2)}]) = 2^{-(n+2)}2^{n} = {1}/{4}.%
\end{equation*}
Now assume to the contrary that $\EC_{\Mis}(\mu_{\infty})$ contains a constant sequence $\PC_n \equiv \PC = \{I_1,\ldots,I_r\}$ with subintervals $I_1 < I_2 < \cdots < I_r$ satisfying $|I_j| < 1$ for $j=1,\ldots,r$. Then there exists $x_* \in (0,1)$, which is a common boundary point of two of these subintervals, say $I_1$ and $I_2$. Let $n_0 = n_0(x_*)$ be chosen according to Step 1. To disprove that $\PC_n\equiv\PC$ is admissible, consider $\ep := 1/8$. Assume to the contrary that there are $\delta>0$ and compact sets $K_{n,j} \subset I_j$ such that%
\begin{enumerate}
\item[(i)] $x \in K_{n,j_1}$ and $y \in K_{n,j_2}$ implies $|x-y| \geq \delta$, if $j_1 \neq j_2$, $n\in\N_0$,%
\item[(ii)] $\mu_n(I_j \backslash K_{n,j}) \leq \ep = 1/8$ for all $n\geq0$.%
\end{enumerate}
Property (i) implies that for each $n$, $I_1 \backslash K_{k_n,1}$ contains $[x_* - \delta/2,x_*]$, or $I_2 \backslash K_{k_n,2}$ contains $[x_*,x_* + \delta/2]$. Choose $n\geq n_0$ large enough that $2^{-(n+2)} < \delta/2$. Then%
\begin{equation*}
  \mu_{k_n}(I_1 \backslash K_{k_n,1}) \geq \mu_{k_n}([x_*-\delta/2,x_*]) \geq \mu_{k_n}([x_*-2^{-(n+2)},x_*]) = {1}/{4} > \ep%
\end{equation*}
or%
\begin{equation*}
  \mu_{k_n}(I_2 \backslash K_{k_n,2}) \geq \mu_{k_n}([x_*,x_*+\delta/2]) \geq \mu_{k_n}([x_*,x_*+2^{-(n+2)}]) = {1}/{4} > \ep,%
\end{equation*}
in contradiction to (ii). Since $|f_n'(x)| \leq 1$ for all $n$, $f_{\infty}$ is equicontinuous.%
\end{proof}

We will now return to the general, possibly non-stationary, sequence of spaces $X_\infty=(X_n)_{n=0}^\infty$. We are in position to prove the following extension of the product theorem (item (iii) in Proposition \ref{prop_finescaleprops}).%

\begin{proposition}\label{prop_productext}
Let the assumptions of Proposition \ref{prop_finescaleprops}(iii) be satisfied. Then, under each of the following three conditions, equality holds in \eqref{eq_productineq}:%
\begin{enumerate}
\item[(a)] $h_{\EC_{\Mis}}(g_{\infty}) = 0$. In this case,%
\begin{equation*}
  h_{\EC_{\Mis}}(f_{\infty} \tm g_{\infty}) = h_{\EC_{\Mis}}(f_{\infty}).%
\end{equation*}
\item[(b)] One of the systems is autonomous, say $(Y_n,g_n,\nu_n) \equiv (Y,g,\nu)$. In this case,%
\begin{equation*}
  h_{\EC_{\Mis}}(f_{\infty}\tm g_{\infty}) = h_{\EC_{\Mis}}(f_{\infty}) + h_{\nu}(g),%
\end{equation*}
where $h_{\nu}(g)$ is Kolmogorov-Sinai entropy of $g$ w.r.t.~$\nu$.%
\item[(c)] $(X_{\infty},f_{\infty}) = (Y_{\infty},g_{\infty})$ and $\mu_{\infty} = \nu_{\infty}$. In this case,%
\begin{equation*}
  h_{\EC_{\Mis}}(f_{\infty} \tm f_{\infty}) = 2h_{\EC_{\Mis}}(f_{\infty}).%
\end{equation*}
\end{enumerate}
\end{proposition}

\begin{proof}
For (a) we only have to show $h_{\EC_{\Mis}}(f_{\infty}\tm g_{\infty}) \geq h_{\EC_{\Mis}}(f_{\infty})$. To this end, consider $\pi_n: X_n \tm Y_n \rightarrow X_n$, $(x,y) \mapsto x$. This sequence is an equi-semiconjugacy from $f_{\infty} \tm g_{\infty}$ to $f_{\infty}$. Then the inequality follows from \cite[Prop.~27]{Ka1}.%

To show (b), let $\PC_{\infty} \in \EC_{\Mis}(\mu_{\infty})$ and let $\QC$ be a finite measurable partition of $Y$. As in the proof of Proposition \ref{prop_finescaleprops}(iii) we conclude%
\begin{equation*}
  h(f_{\infty} \tm g;\PC_{\infty} \tm \QC) = \limsup_{n\rightarrow\infty}\frac{1}{n}\Big[H_{\mu_0}\Big(\bigvee_{i=0}^{n-1}f_0^{-i}\PC_i\Big) + H_{\nu}\Big(\bigvee_{i=0}^{n-1}g^{-i}\QC\Big)\Big].%
\end{equation*}
Using the inequality $\limsup_n (a_n + b_n) \geq \limsup_n a_n + \liminf_n b_n$, we obtain%
\begin{equation*}
  h(f_{\infty} \tm g;\PC_{\infty} \tm \QC) \geq \limsup_{n\rightarrow\infty}\frac{1}{n}H_{\mu_0}\Big(\bigvee_{i=0}^{n-1}f_0^{-i}\PC_i\Big) + \lim_{n\rightarrow\infty}\frac{1}{n}H_{\nu}\Big(\bigvee_{i=0}^{n-1}g^{-i}\QC\Big),%
\end{equation*}
because for an autonomous system the limit in the above expression exists. Now let $(\PC_{\infty}^k)_{k=0}^{\infty}$ and $(\QC^k)_{k=0}^{\infty}$ be fine-scale sequences, where $\PC_{\infty}^k \in \EC_{\Mis}(\mu_{\infty})$. Then, from the above it follows that%
\begin{align*}
  h_{\EC_{\Mis}(\mu_{\infty}\tm\nu)}(f_{\infty} \tm g) &= \lim_{k\rightarrow\infty}h(f_{\infty}\tm g;\PC_{\infty}^k \tm \QC^k)\\
	&\geq \lim_{k\rightarrow\infty}h(f_{\infty};\PC_{\infty}^k) + \lim_{k\rightarrow\infty}h(g;\QC^k)\\
	&= h_{\EC_{\Mis}(\mu_{\infty})}(f_{\infty}) + h_{\nu}(g),%
\end{align*}
where we use Corollary \ref{cor_autonomous}. Together with Proposition \ref{prop_finescaleprops}(iii), (b) is proved.%

Finally, to prove (c), note that for any admissible $\PC_{\infty}$ we have%
\begin{equation*}
  h(f_{\infty}\tm f_{\infty};\PC_{\infty} \tm \PC_{\infty}) =
  2\limsup_{n\rightarrow\infty}\frac{1}{n}H_{\mu_0}\Big(\bigvee_{i=0}^{n-1}f_0^{-i}\PC_i\Big) = 2 h(f_{\infty};\PC_{\infty}).%
\end{equation*}
By considering a fine-scale sequence $(\PC_{\infty}^k)_{k=0}^{\infty}$, the assertion easily follows.%
\end{proof}

In the classical theory sometimes it is sufficient to compute the entropy on a single partition, which is therefore called a \emph{generator} of the entropy. This is the case, e.g., if the given map is expansive. In \cite{Ka2} we introduced the notion of \emph{strong uniform expansivity} for nonautonomous systems. The next result shows that this condition guarantees the existence of a generating sequence of partitions, provided that the IMS is sufficiently nice.%

\begin{corollary}
Let $(X_{\infty},f_{\infty})$ be a topological NDS which is strongly uniformly expansive, i.e., there is $\delta>0$ such that for any $\ep>0$ there is $N\in\N$ with%
\begin{equation*}
  \max_{0\leq i\leq N}d_{n+i}(f_n^i(x),f_n^i(y)) < \delta \quad\Rightarrow\quad d_n(x,y) < \ep%
\end{equation*}
for all $n\geq0$. Let $\mu_{\infty}$ be an IMS for $f_{\infty}$ and assume that $\EC_{\Mis}(\mu_{\infty})$ contains a sequence $\PC_{\infty}$ with $\diam\PC_{\infty} < \delta$. Then $\mu_{\infty}$ is a fine-scale IMS and%
\begin{equation*}
  h(f_{\infty},\mu_{\infty}) = h(f_{\infty},\mu_{\infty};\PC_{\infty}).%
\end{equation*}
\end{corollary}

\begin{proof}
We consider the sequence $\PC^{\langle k \rangle}_{\infty}$, $\PC^{\langle k \rangle}_n = \bigvee_{i=0}^{k-1}f_n^{-i}\PC_{n+i}$. According to \cite[Def.~15 and Prop.~26]{Ka1}, $\PC^{\langle k \rangle}_{\infty} \in \EC_{\Mis}$ and by \cite[Prop.~9(v)]{Ka1}, $h(f_{\infty};\PC^{\langle k \rangle}_{\infty}) = h(f_{\infty};\PC_{\infty})$. Hence, if $\diam\PC^{\langle k \rangle}_{\infty} \rightarrow 0$ for $k\rightarrow\infty$, then Theorem \ref{thm_mecomp} yields the result. Let $\emptyset \neq P \in \PC^{\langle k \rangle}_n$ for some $k$ and $n$. Fix arbitrary $x,y\in P$. Then, for $i\in\{0,1,\ldots,k-1\}$, $f_n^i(x),f_n^i(y)$ are contained in the same element of $\PC_{n+i}$, implying $d_{n+i}(f_n^i(x),f_n^i(y)) < \delta$ for $i=0,\ldots,k-1$. For $k = k(\ep)$ large enough, this implies $d_n(x,y) < \ep$. Hence, $\diam\PC^{\langle k\rangle}_{\infty}<\ep$, completing the proof.%
\end{proof}

\section{Nonstationary subshifts of finite type}\label{sec_nsfts}%

In this section, we consider the time-dependent analogues of subshifts of finite type (topological Markov chains) studied in \cite{AFi,Fis} by Arnoux and Fisher. We start with some definitions.%

Let $(\AC_i)_{i=0}^{\infty}$ be a sequence of finite alphabets $\AC_i = \{1,\ldots,l_i\}$. Consider also a sequence of $0$-$1$-matrices $(L_i)_{i=0}^{\infty}$ of corresponding dimensions $l_i \tm l_{i+1}$. For any nonnegative integers $n\leq m$ we write%
\begin{equation*}
  w(n,m) := \#\left\{ (x_n,\ldots,x_m)\ :\ x_i \in \AC_i \mbox{ and } (L_i)_{x_ix_{i+1}} = 1,\ n \leq i \leq m-1 \right\}.%
\end{equation*}
Furthermore, we make the following assumptions:%
\begin{enumerate}
\item[(A1)] The sequence $(L_i)_{i=0}^{\infty}$ is reduced, meaning that none of the rows or columns of each $L_i$ is identically zero (cf.~\cite[Def.~2.1]{Fis}).%
\item[(A2)] The sequence $(l_i)_{i=0}^{\infty}$ of alphabet sizes is bounded.%
\end{enumerate}
It is easy to see that the following properties are satisfied:
\begin{enumerate}
\item[(W1)] $w(n,n) = l_n$.%
\item[(W2)] $w(n,m+1) \geq w(n,m)$ for all $m\geq n$.%
\item[(W3)] $w(0,n+m) \leq w(0,n)w(n+1,n+m)$ for $n\geq0$, $m\geq1$.%
\end{enumerate}

Now we define a sequence of spaces $\Sigma^{k,+}_{(L)}$, $k\geq0$, by%
\begin{equation*}
  \Sigma^{k,+}_{(L)} := \left\{ (x_k,x_{k+1},x_{k+2},\ldots)\ :\ x_i \in \AC_i \mbox{ and } (L_i)_{x_ix_{i+1}} = 1,\ \forall i\geq k \right\},%
\end{equation*}
and a sequence of maps by%
\begin{equation*}
  \sigma_k:\Sigma^{k,+}_{(L)} \rightarrow \Sigma^{k+1,+}_{(L)},\quad \sigma_k(x_k,x_{k+1},x_{k+2},\ldots) = (x_{k+1},x_{k+2},x_{k+3},\ldots).%
\end{equation*}
That is, $\sigma_k$ is the usual left shift of sequences, restricted to $\Sigma^{k,+}_{(L)}$. On $\Sigma^{k,+}_{(L)}$ we define a metric $d_k$ by%
\begin{equation*}
  d_k(x,y) := \left\{\begin{array}{cl}
	                               0 & \mbox{ if } x = y\\
																 1 & \mbox{ if } x_0 \neq y_0\\
															w(k,l)^{-1} & \mbox{ if } l = \max\{ r \geq k : x_i = y_i,\ i = k,\ldots,r \}%
										 \end{array}\right.%
\end{equation*}
for any two $x = (x_k,x_{k+1},\ldots)$ and $y = (y_k,y_{k+1},\ldots)$ in $\Sigma^{k,+}_{(L)}$. The proof of the following proposition uses standard arguments and will be omitted. We just note that for the equicontinuity of $\sigma_{\infty}$ assumption (A2) is essential.%

\begin{proposition}\label{prop_nsft}
$(\Sigma^{k,+}_{(L)},d_k)$ is a compact metric space and $\sigma_k$ is surjective. Moreover, the sequence $\sigma_{\infty} := (\sigma_k)_{k=0}^{\infty}$ is equicontinuous.%
\end{proposition}

\begin{definition}
Following Arnoux and Fisher \cite{AFi}, we call the topological NDS $(\Sigma^{\infty,+}_{(L)},\sigma_{\infty})$ a \emph{nonstationary subshift of finite type}, or briefly an NSFT. Moreover, we introduce the notation $  L^{(k,n)} := L_k L_{k+1} \cdots L_{n-1} $ for $n > k$.\end{definition}

\subsection{Topological entropy}%

In this subsection, we give a formula for the topological entropy of an NSFT. The proof of the following lemma is completely analogous to the stationary case.%

\begin{lemma}\label{lem_noadmwords}
For each $n\geq1$, the number of admissible words $(x_0,\ldots,x_n)$ with $x_i \in \AC_i$ such that $x_0 = \alpha$ and $x_n = \beta$ for fixed $\alpha \in \AC_0$ and $\beta \in \AC_n$, is equal to the matrix entry $L^{(0,n)}_{\alpha\beta}$.
\end{lemma}

\begin{definition}
Given $\ep>0$, define for each $n\geq0$%
\begin{equation}\label{eq_defmk}
  m_n(\ep) := \min\left\{ m\geq0\ :\ w(n,n+m) > \frac{1}{\ep} \right\}.%
\end{equation}
\end{definition}

\begin{theorem}\label{thm_te_formula}
The topological entropy of an NSFT is given by%
\begin{equation}\label{eq_topent_form}
  h_{\tp}(\sigma_{\infty}) = \limsup_{n\rightarrow\infty}\frac{1}{n}\log w(0,n) = \limsup_{n\rightarrow\infty}\frac{1}{n}\log\big\|L^{(0,n)}\big\|_{(n)},%
\end{equation}
where $\|\cdot\|_{(n)}$ denotes the norm on $\R^{l_0 \tm l_n}$ given by $\|A\|_{(n)} = \sum_{\alpha,\beta}|A_{\alpha\beta}|$.%
\end{theorem}

\begin{proof}
If $\Sigma^{0,+}_{(L)}$ is finite, the assertion is trivial, and hence we assume $\#\Sigma^{0,+}_{(L)} = \infty$. The existence of $m_n(\ep)$ for every $\ep$ is equivalent to $\lim_{k\rightarrow\infty}w(n,n+k) = \infty$, which is equivalent to $\#\Sigma^{n,+}_{(L)} = \infty$. By assumption, this holds for $n=0$. Because $w(0,n+k) \leq w(0,n-1)w(n,n+k)$, it holds for every $n$. We first prove%
\begin{equation}\label{eq_topent_preform}
  h_{\tp}(\sigma_{\infty}) = \lim_{\ep\searrow0}\limsup_{n\rightarrow\infty}\frac{1}{n}\log w(0,n+m_n(\ep)).%
\end{equation}
For fixed $n\in\N$ and $\ep>0$ let $m_n = m_n(\ep)$ and consider a set $E \subset \Sigma^{0,+}_{(L)}$ containing for every admissible word $(x_0,\ldots,x_{n+m_n})$ a unique $y \in \Sigma^{0,+}_{(L)}$ with $y_i = x_i$ for $0 \leq i \leq n+m_n$. Then $\# E = w(0,n+m_n)$ and any two $x \neq y$ in $E$ must differ in at least one of their first $n+m_n+1$ components. The maximum%
\begin{equation*}
  d_{0,n}(x,y) = \max_{0 \leq j \leq n}d_j(\sigma_0^j(x),\sigma_0^j(y))%
\end{equation*}
is equal to $1$ if $x$ and $y$ differ in one of their first $n+1$ components. If the $(n+l)$-th component ($1 \leq l \leq m_n+1$) is the first in which they differ, then%
\begin{equation*}
  d_j(\sigma_0^j(x),\sigma_0^j(y)) = \frac{1}{w(j,n+l-1)} \geq \frac{1}{w(j,n+m_n)}.%
\end{equation*}
Consequently,%
\begin{equation*}
  d_{0,n}(x,y) \geq \max_{0\leq j\leq n}\frac{1}{w(j,n+m_n)} = \frac{1}{w(n,n+m_n)},%
\end{equation*}
which easily follows from assumption (A1). This inequality is also true in the case, where $x$ and $y$ differ in one of their first $n+1$ components. By definition of $m_n(\ep)$, it follows that $w(n,n+m_n-1) \leq 1/\ep$ in the case where $m_n(\ep) > 0$ (which we assume w.l.o.g., since otherwise $w(n,n+m_n) = l_n$, which is bounded, and we are only interested in small $\ep$). Thus,%
\begin{equation*}
  d_{0,n}(x,y) \geq \frac{1}{w(n,n+m_n)} \geq \frac{1}{w(n,n+m_n-1)l_{n+m_n}} \geq \frac{\ep}{\sup_n l_n},%
\end{equation*}
implying that $E$ is an $(n,\ep/(\sup_n l_n))$-separated set, and hence%
\begin{equation*}
  h_{\tp}(\sigma_{\infty}) \geq \lim_{\ep\searrow0}\limsup_{n\rightarrow\infty}\frac{1}{n}\log w(0,n+m_n(\ep)).%
\end{equation*}
Choosing $y\in E$ for a given $x\in\Sigma^{0,+}_{(L)}$ such that $x_i = y_i$ for $i=0,1,\ldots,n+m_n$, we find%
\begin{equation*}
  d_{0,n}(x,y) = \max_{0\leq j \leq n}d_j(\sigma_0^j(x),\sigma_0^j(y)) \leq \frac{1}{w(n,n+m_n)} < \ep,%
\end{equation*}
and hence $E$ is $(n,\ep)$-spanning, implying the inequality ``$\leq$'' in \eqref{eq_topent_preform}. Note that%
\begin{eqnarray*}
  w(0,n+m_n(\ep)-1) \leq w(0,n-1)w(n,n+m_n(\ep)-1) \leq w(0,n)\frac{1}{\ep},%
\end{eqnarray*}
following from (W3). Consequently, we obtain%
\begin{eqnarray*}
  && \limsup_{n\rightarrow\infty}\frac{1}{n}\log w(0,n+m_n(\ep))\\
	&& \qquad\qquad \leq \limsup_{n\rightarrow\infty}\frac{1}{n}\log\left(w(0,n+m_n(\ep)-1) l_{n+m_n(\ep)}\right)\\
	&& \qquad\qquad \leq \limsup_{n\rightarrow\infty}\frac{1}{n}\log\left(\frac{1}{\ep}l_{n+m_n(\ep)} w(0,n)\right).%
\end{eqnarray*}
Since $(l_n)_{n=0}^{\infty}$ is bounded by assumption (A2), the inequality ``$h_{\tp}(\sigma_{\infty}) \leq \ldots$'' in \eqref{eq_topent_form} follows. The other inequality trivially follows from \eqref{eq_topent_preform}. The second equality in \eqref{eq_topent_form} is an easy consequence of Lemma \ref{lem_noadmwords}.%
\end{proof}

\begin{remark}
\noindent{\em (i)}\, In the second expression for $h_{\tp}(\sigma_{\infty})$ in \eqref{eq_topent_form} we can also use the operator norm $\|\cdot\|_{1,(n)}$ coming from the vector norm $\|x\| = \sum_i |x_i|$, for instance. This follows from the equivalence $\|A\|_{1,(n)} \leq \|A\|_{(n)} \leq l_n\|A\|_{1,(n)}$ and boundedness of $l_n$. Indeed, any of the usual operator norms will do, because they are all equivalent to each other, and the equivalence factors depend on the dimensions of the spaces, which vary within a finite set.%

\noindent{\em (ii)}\, Since $\sigma_{\infty}$ is an equicontinuous sequence of surjective maps, we know from Proposition \ref{prop_te_timeindep} that $h_{\tp}(\sigma_{k,\infty}) = h_{\tp}(\sigma_{l,\infty})$ for any $k,l\geq0$. From the formula of the above theorem this can also be seen directly. Indeed, it holds that $w(1,n) \leq w(0,n) \leq l_0w(1,n)$, where the first inequality is trivial and the second one follows from (W3). This implies%
\begin{equation*}
  \limsup_{n\rightarrow\infty} \frac{1}{n}\log w(0,n) = \limsup_{n\rightarrow\infty} \frac{1}{n}\log w(1,n) = \limsup_{n\rightarrow\infty} \frac{1}{n}\log w(1,n+1),%
\end{equation*}
and hence $h_{\tp}(\sigma_{0,\infty}) = h_{\tp}(\sigma_{1,\infty})$. The rest follows inductively.%

\noindent{\em (iii)}\, In particular, the entropy formula shows that $h_{\tp}(\sigma_{\infty})$ is finite, since $w(0,k) \leq \prod_{i=0}^{k-1}l_i \leq L^k$ with $L := \sup_{k\geq0}l_k$, and hence $h_{\tp}(\sigma_{\infty}) \leq \log L$.%
\end{remark}

\subsection{Metric entropy}%

Now we turn to the computation of the metric entropy with respect to an appropriately defined IMS. The definition is taken from Fisher \cite{Fis} who writes:%

\emph{\ldots we shall find that many things carry over more or less directly from the stationary case. For example, the measure of maximal entropy of a subshift of finite type has a very simple formula discovered by Shannon in an information theory context (\ldots), involving the left and right Perron-Frobenius eigenvectors of the matrix (the proof that this maximizes entropy within the Markov measures is due to Shannon, while Parry in [\ldots] showed that Shannon's measure gives the maximum over all invariant probability measures; \ldots We follow standard usage in calling this the Parry measure. Here (without worrying about how to define entropy for sequences of maps) we simply replace these eigenvectors by eigenvector sequences, and define a nonstationary Markov chain which gives us an exact analogue of the Parry measure. The formula gives us a sequence of measures on the components, invariant in the sense that one is carried to the next by the shift map.}

Hence, Fisher defines an IMS for an NSFT, which is constructed analogously to the Parry measure for classical subshifts. The aim of this subsection is to compute the metric entropy of this IMS and to compare it with the topological entropy.%

Given $x_i \in \AC_i$ for $k \leq i \leq m$, we define the \emph{cylinder set}%
\begin{equation*}
  [.x_k \ldots x_m] := \left\{ y = (y_k,y_{k+1},\ldots) \in \Sigma^{k,+}_{(L)}\ :\ y_i = x_i,\ i = k,\ldots,m \right\}.%
\end{equation*}
We also define for each $0 \leq k \leq m$ the partition%
\begin{equation*}
  \PC_k^m := \left\{ [.x_k \ldots x_m]\ :\ x_i \in \AC_i \mbox{ and } (L_i)_{x_ix_{i+1}} = 1,\ k \leq i \leq m-1 \right\}.%
\end{equation*}

\begin{proposition}\label{prop_cylinderpartitions}
Let $\mu_{\infty}$ be an IMS for $\sigma_{\infty}$ and let $m_k(\ep)$ be given by \eqref{eq_defmk}. Then, for any $\ep>0$ the sequence $\PC_{\infty,\ep} = (\PC_{k,\ep})_{k=0}^{\infty}$, $\PC_{k,\ep} := \PC_k^{k+m_k(\ep)}$, is admissible and $\diam\PC_{\infty,\ep} < \ep$. Hence, every IMS for $\sigma_{\infty}$ is a fine-scale IMS.%
\end{proposition}

\begin{proof}
The number of elements in $\PC_{k,\ep}$ is equal to $w(k,k+m_k(\ep))$ and $w(k,k+m_k(\ep)) \leq w(k,k+m_k(\ep)-1)l_{k+m_k(\ep)} \leq (\sup_k l_k)/\ep$. It is clear that $\PC_{k,\ep}$ is a partition of $\Sigma^{k,+}_{(L)}$. It is easy to see that the cylinder sets are closed and hence compact. If $x,y \in \Sigma^{k,+}_{(L)}$ are in two distinct elements of $\PC_{k,\ep}$, these sequences differ in at least one of their first $m_k(\ep)+1$ components, and hence%
\begin{equation*}
  d_k(x,y) \geq \frac{1}{w(k,k+m_k)} \geq \frac{1}{w(k,k+m_k(\ep)-1)l_{k+m_k(\ep)}} \geq \frac{\ep}{\sup_k l_k}.%
\end{equation*}
Obviously, this implies $\PC_{\infty,\ep} \in \EC_{\Mis}(\mu_{\infty})$. For the distance of two points $x,y\in\PC_{k,\ep}$ we obtain $d_k(x,y) \leq w(k,k+m_k(\ep))^{-1} < \ep$, concluding the proof.%
\end{proof}

With Theorem \ref{thm_mecomp} we obtain as an immediate consequence:%

\begin{corollary}\label{cor_nsft_me}
For every IMS $\mu_{\infty}$ of $\sigma_{\infty}$ we have%
\begin{equation*}
  h(\sigma_{\infty},\mu_{\infty}) = \lim_{\ep\searrow0}h(\sigma_{\infty},\mu_{\infty};\PC_{\infty,\ep}).%
\end{equation*}
\end{corollary}

The following definition, (its first part is taken from \cite[Def.~2.3]{Fis}) generalizes the concept of primitivity for single matrices.%

\begin{definition}\label{defn3.21}
The sequence $(L_i)_{i=0}^{\infty}$ is called \emph{primitive} if for each $i\geq0$ there exists $n>i$ such that all entries of $L^{(i,n)}$ are strictly positive. We denote by $N_i$ the least integer $n$ such that all entries of $L^{(i,n)}$ are strictly positive and call the sequence $(L_i)_{i=0}^{\infty}$ \emph{uniformly primitive} if the sequence $(N_i)_{i=0}^\infty$ is bounded.%
\end{definition}

In the rest of the subsection, additionally to (A1) and (A2), we assume that%
\begin{enumerate}
\item[(A3)] The sequence $(L_i)_{i=0}^{\infty}$ is primitive.%
\end{enumerate}

\begin{definition}
Given a sequence $(A_i)_{i=0}^{\infty}$ of $l_i \tm l_{i+1}$ nonnegative real matrices, a sequence $(v_i)_{i=0}^{\infty}$ of nonzero column vectors satisfying $A_iv_{i+1} = \lambda_i v_i$ or row vectors satisfying $v_i^t A_i = \lambda_i v_{i+1}^t$, with nonzero numbers $\lambda_i$, for all $i\geq0$ is called a \emph{column} or \emph{row eigenvector sequence} with \emph{eigenvalues} $\lambda_i$.%
\end{definition}

We write $C_i^+$ for the positive cone in $\R^{l_i}$, i.e., the set of all column vectors with nonnegative entries only. For the matrix sequence $(L_i)_{i=0}^{\infty}$ we define%
\begin{equation*}
  \widehat{\Omega}_{(L)} := \left\{ \widehat{\mathbf{w}}_\infty = (\widehat{\mathbf{w}}_i)_{i=0}^\infty \mbox{ with } \widehat{\mathbf{w}}_i = L_i\widehat{\mathbf{w}}_{i+1} \mbox{ and } \widehat{\mathbf{w}}_i \in \inner C_i^+ \right\}.%
\end{equation*}
Analogously, we write $R_i^+$ for the positive cone in $^t\R^{l_i}$, i.e., the set of all row vectors with nonnegative entries only, and we define%
\begin{equation*}
  \widehat{\Omega}_{(L)}^t := \left\{ \widehat{\mathbf{v}}^t_\infty = (\widehat{\mathbf{v}}_i^t)_{i=0}^\infty \mbox{ with } \widehat{\mathbf{v}}_i^t L_i = \widehat{\mathbf{v}}_{i+1}^t \mbox{ and } \widehat{\mathbf{v}}_i \in \inner R_i^+ \right\}.%
\end{equation*}

The following lemma can be found in \cite[Lem.~4.2]{Fis}.%

\begin{lemma}\label{lem_eigenvectorseqs}
For a reduced and primitive sequence $(L_i)_{i=0}^{\infty}$ the sets $\widehat{\Omega}_{(L)}$ and $\widehat{\Omega}_{(L)}^t$ are nonempty. A sequence in $\widehat{\Omega}_{(L)}^t$ is determined by choosing a strictly positive first element $\widehat{\vB}_0^t$. If $(\widehat{\mathbf{w}}_i)_{i=0}^\infty\in\widehat{\Omega}_{(L)}$, then $\|\widehat{\mathbf{w}}_i\|\ge\|\widehat{\mathbf{w}}_{i+1}\|$ for $i\ge0$.%
\end{lemma}

We take some $\widehat{\wB}_\infty\in\widehat{\Omega}_{(L)}$ and project each component $\widehat{\wB}_i$ to the unit simplex:%
\begin{equation*}
  \wB_i := \frac{\widehat{\wB}_i}{\|\widehat{\wB}_i\|},\quad \|\wB\| = \sum_{\alpha} |(\wB)_{\alpha}|.%
\end{equation*}
We then normalize the sequence of row vectors in a different way which depends on the choice of $\widehat{\wB}_\infty$: we define $(\vB_i^t)_{i=0}^\infty$ by%
\begin{equation*}
  \vB_i^t := \frac{\widehat{\vB}_i^t}{\widehat{\vB}_i^t\wB_i},\quad i\ge0.%
\end{equation*}
Since the entries of $\wB_i$ are strictly positive, we are not dividing by zero. We define real numbers $\lambda_i = \lambda_i(\widehat{\wB}_\infty)$ by  $\lambda_i:= \|\widehat{\wB}_i\|/\|\widehat{\wB}_{i+1}\| \geq 1$, and we have%
\begin{equation*}
  L_i\wB_{i+1} = \lambda_i\wB_i \mbox{\quad for all\ } i\geq0.%
\end{equation*}
A direct calculation as in \cite[Lem.~4.3]{AFi} yields the following fact.%

\begin{lemma}
Let $(\widehat{\wB}_i)_{i=0}^{\infty} \in \widehat{\Omega}_{(L)}$ and $(\widehat{\vB}^t_i)_{i=0}^{\infty} \in \widehat{\Omega}_{(L)}^t$, and let $(\wB_i)_{i=0}^{\infty}$ and $(\vB_i^t)_{i=0}^{\infty}$ be the corresponding normalized sequences. Then $(\vB^t_i)_{i=0}^{\infty}$ has the same eigenvalues as $(\wB_i)_{i=0}^{\infty}$.%
\end{lemma}

From a choice $\widehat{\wB}_i$, $\widehat{\vB}_i^t$ we now define a nonstationary Parry measure. First we define a sequence of row vectors $\mathbf{\pi}^t_\infty = (\mathbf{\pi}_i^t)_{i=0}^{\infty}$ by 
\begin{equation}\label{defpi}
  (\mathbf{\pi}_i^t)_\alpha := (\vB_i)_\alpha(\wB_i)_\alpha,\quad \alpha=1,\ldots,l_i,\quad i\geq 0.%
\end{equation} 
The entries of $\mathbf{\pi}_i^t$ are strictly positive and $\sum_{\alpha}(\mathbf{\pi}_i^t)_{\alpha} = 1$. Then we put $P_i := (1/\lambda_i)W_i^{-1}L_iW_{i+1}$, where $W_i$ is an $(l_i \tm l_i)$-diagonal matrix with the entries of $\mathbf{w}_i$ on the diagonal. Each $P_i$ is a stochastic matrix and $\mathbf{\pi}_i^t P_i =\mathbf{\pi}_{i+1}^t$ for $i\ge0$. For $k \leq m$ we put%
\begin{equation}\label{dfnlambdpr}
  \lambda^{(k,m)} = \lambda^{(k,m)}_{(\wB)} := \left\{\begin{array}{cc}
	                                 1 & \mbox{if } k = m\\
																	 \prod_{i=k}^{m-1}\lambda_i(\mathbf{w}) & \mbox{otherwise}%
																\end{array}\right.,%
\end{equation}
and notice that%
\begin{equation}\label{eq_eigenv_eqs}
  L^{(0,i)}\mathbf{w}_i = \lambda^{(0,i)}_{(\mathbf{w})}\mathbf{w}_0 \mbox{\quad and\quad } \mathbf{v}_0^t L^{(0,i)}=\lambda^{(0,i)}_{(\mathbf{w})}
\mathbf{v}_i^t,\quad i\geq 1.%
\end{equation}
A probability measure on $\Sigma^{k,+}_{(L)}$ is defined on cylinder sets by%
\begin{equation*}
\begin{split}
  \mu_k([.x_k\ldots x_m]) :&= (\mathbf{\pi}^t_k)_{x_k}(P_k)_{x_kx_{k+1}} \cdots (P_{m-1})_{x_{m-1}x_m}\\& = \frac{1}{\lambda^{(k,m)}_{(\mathbf{w})}}(\vB_k^t)_{x_k}(\wB_m)_{x_m}\end{split}
\end{equation*}
and extended uniquely by $\sigma$-additivity. Any such $\mu_{\infty}$ is called a \emph{Parry IMS}.%

In what follows, we say that $\gamma\in\AC_k$ maximizes the sum $\sum_{\alpha}L^{(0,k)}_{\alpha\gamma}$ if 
\[\sum_{\alpha}L^{(0,k)}_{\alpha\gamma}=\max\big\{\sum_{\alpha}L^{(0,k)}_{\alpha\beta}: \beta\in\AC_{k}
\big\}.\]%
We recall that $m_k(\ep)=\min\{m\geq0: w(k,k+m)>\ep^{-1}\}$ from \eqref{eq_defmk} and introduce the notation%
\begin{equation}\label{dfnnn}
  n_k(\ep) = k+m_k(\ep),\quad \widetilde{n}_k(\ep) = \max_{0 \leq i \leq k}(i + m_i(\ep))
\,\text{ for $k\ge0$ and $\ep>0$}.
\end{equation}

Our main results about the metric entropy of NSFT's are the following.
\begin{theorem}\label{thm_nsft_me1}
Let $(\Sigma^{\infty,+}_{(L)},\sigma_{\infty})$ be an NSFT and $\mu_{\infty}$ an associated Parry IMS.%
\begin{enumerate}
\item[(i)] The following inequalities hold:%
\begin{equation}\label{eq_nsft_entest}
   \limsup_{k\rightarrow\infty}\frac{1}{k}\log\lambda^{(0,k)} \leq h(\sigma_{\infty},\mu_{\infty}) \leq \limsup_{k\rightarrow\infty}\frac{1}{k}\log\|L^{(0,k)}\|_{(k)}.%
\end{equation}
\item[(ii)] The inequalities in \eqref{eq_nsft_entest} are equalities, and hence $h(\sigma_{\infty},\mu_{\infty}) = h_{\tp}(\sigma_{\infty})$, provided the following assumption holds: There exist $\ep>0$ and $k_0\in\N$ such that for all $k\geq k_0$ and $m$ such that $w(k,k+m)>1/\ep$ for each $\gamma\in\AC_k$ which maximizes $\sum_{\alpha}L^{(0,k)}_{\alpha\gamma}$ the inequality $L^{(k,k+m)}_{\gamma\beta} \geq 1$ holds for all $\beta\in\AC_{k+m}$.%
\item[(iii)] Another sufficient condition for the equalities in \eqref{eq_nsft_entest} is%
\begin{equation*}
  \limsup_{k\rightarrow\infty}\frac{1}{k}\log\lambda^{(k,k+N_k)} = 0,%
\end{equation*}
where $N_k$ is the smallest integer $n$ such that all entries of $L^{(k,k+n)}$ are strictly positive. In particular, this holds if $(L_i)_{i=0}^{\infty}$ is uniformly primitive.%
\end{enumerate}
\end{theorem}

\begin{proof}
(i) The second inequality in \eqref{eq_nsft_entest} follows from \eqref{eq_varineq} and Theorem \ref{thm_te_formula}. To show the first one, we first observe that the sequence $(\PC_k^k)_{k=0}^{\infty}$ of cylinder partitions is admissible, since $\#\PC_k^k = l_k$ is bounded and the distance between two elements $[.x_k]$ and $[.y_k]$ of $\PC_k^k$, $x_k \neq y_k$, is $1$. Hence,%
\begin{align*}
 H_{\mu_0}&\big(\bigvee_{i=0}^n\sigma_0^{-i}\PC_i^i\big) = H_{\mu_0}\left(\PC_0^n\right)\\
	& = \sum_{[.x_0\ldots x_n]}\frac{1}{\lambda^{(0,n)}}(\vB_0^t)_{x_0}(\wB_n)_{x_n}\log\Big(\frac{\lambda^{(0,n)}}{(\vB_0^t)_{x_0}(\wB_n)_{x_n}}\Big) \geq \log\frac{\lambda^{(0,n)}}{\max_{\alpha}(\vB_0^t)_{\alpha}},%
\end{align*}
where we use that the components of $\wB_n$ are bounded by $1$ and $\mu_0$ is a probability measure. Dividing by $n$ and letting $n$ go to infinity  gives the desired estimate.%

(ii) It suffices to show that%
\begin{equation*}
  \limsup_{k\rightarrow\infty}\frac{1}{k}\log\lambda^{(0,k)} \geq \limsup_{k\rightarrow\infty}\frac{1}{k}\log\|L^{(0,k)}\|_{(k)}.%
\end{equation*}
To this end, let $n_k = n_k(\ep)$ and observe that $L^{(0,n_k)}\wB_n = \lambda^{(0,n_k)}\wB_0$ yields%
\begin{equation*}
  \lambda^{(0,n_k)} = \sum_{\alpha,\beta}L^{(0,n_k)}_{\alpha\beta}(\wB_{n_k})_{\beta}
	                  = \sum_{\alpha,\gamma,\beta}L^{(0,k)}_{\alpha\gamma}L^{(k,n_k)}_{\gamma\beta}(\wB_{n_k})_{\beta}.%
\end{equation*}
Choose $\gamma'$ that maximizes $\sum_{\alpha}L^{(0,k)}_{\alpha\gamma'}$. Then, using the assumption,%
\begin{equation*}
  \lambda^{(0,n_k)} \geq \sum_{\alpha}L^{(0,k)}_{\alpha\gamma'}\sum_{\beta}(\wB_{n_k})_{\beta} = \big\|L^{(0,k)}\big\|_1,%
\end{equation*}
where $\|\cdot\|_1$ is the operator norm derived from the $\ell^1$-vector norm. Writing $\|\cdot\|_{(n)}$ for the sum norm, this implies%
\begin{eqnarray*}
  \lambda^{(0,k)} &=& \frac{\lambda^{(0,n_k)}}{\lambda^{(k,n_k)}} \geq \frac{\|L^{(0,k)}\|_1}{\|L^{(k,n_k)}\|_{(n_k)}} = \frac{\|L^{(0,k)}\|_1}{w(k,n_k)}\\
	&\geq& \frac{1}{w(k,n_k-1)l_{n_k}} \|L^{(0,k)}\|_1  \geq \frac{\ep}{\sup_i l_i} \|L^{(0,k)}\|_1.%
\end{eqnarray*}
Using that all norms on a finite-dimensional space are equivalent and the equivalence factors stay bounded when the dimensions vary within a finite set, we obtain the assertion.%

(iii) From $L^{(0,k+N_k)}\wB_{k+N_k} = L^{(0,k)}L^{(k,k+N_k)}\wB_{k+N_k} = \lambda^{(0,k+N_k)}\wB_0$ we get%
\begin{eqnarray*}
  \lambda^{(0,k+N_k)} &=& \lambda ^{(0,k+N_k)}\|\wB_0\| = \sum_{\alpha,\gamma,\beta}L^{(0,k)}_{\alpha\gamma}\underbrace{L^{(k,k+N_k)}_{\gamma\beta}}_{\geq1}(\wB_{k+N_k})_{\beta}\\
	&\geq& \|L^{(0,k)}\|_{(k)}\|\wB_{k+N_k}\| = \|L^{(0,k)}\|_{(k)}.%
\end{eqnarray*}
This implies the first assertion in (iii). If $(L_i)_{i=0}^{\infty}$ is uniformly primitive then the sequence $(N_k)_{k=0}^{\infty}$ is bounded, and since $\lambda_i \leq \|L_i\|_{(i)} \leq (\sup_i l_i)^2$, this implies that the sequence $(\lambda^{(k,k+N_k)})_{k=0}^{\infty}$ is bounded.%
\end{proof}

\begin{theorem}\label{thm_nsft_me2}
Let $(\Sigma^{\infty,+}_{(L)},\sigma_{\infty})$ be an NSFT and $\mu_{\infty}$ an associated Parry IMS.%
\begin{enumerate}
\item[(i)] With the notation \eqref{defpi} and \eqref{dfnnn}, we have%
\begin{equation}\label{eq_me_form}
  h(\sigma_{\infty},\mu_{\infty}) = \lim_{\ep\searrow0}\limsup_{k\rightarrow\infty}\frac{1}{k}\sum_{\beta}(\pi_{\widetilde{n}_k(\ep)})_{\beta}\log\sum_{\alpha}L^{(0,\widetilde{n}_k(\ep))}_{\alpha\beta}.%
\end{equation}
\item[(ii)] The following condition is sufficient for the equality of $h(\sigma_{\infty},\mu_{\infty})$ and $h_{\tp}(\sigma_{\infty})$: For every $\delta\in(0,1)$ there exist $\ep>0$ and $k_0\in\N$ such that for all $k\geq k_0$ and each $\gamma\in\AC_k$ which maximizes $\sum_{\alpha}L^{(0,k)}_{\alpha\gamma}$ the inequality $L^{(k,n_k(\ep))}_{\gamma\beta}\geq 1$ holds for all $\beta$ in a set $B\subset\AC_{n_k(\ep)}$ satisfying $\sum_{\beta\in B}(\pi_{n_k(\ep)})_{\beta} \geq 1 - \delta$.%
\end{enumerate}
\end{theorem}

\begin{proof}
(i) Let us consider the cylinder partition $\PC_0^k$. Without loss of generality, we may assume that $\vB_0^t$ is the vector all of whose entries are ones. Using Lemma \ref{lem_noadmwords} and \eqref{eq_eigenv_eqs}, we compute%
\begin{align}
 H_{\mu_0}(\PC_0^k)&= \sum_{[.x_0\ldots x_k]}\frac{1}{\lambda^{(0,k)}}(\vB_0^t)_{x_0}(\wB_k)_{x_k}\log\frac{\lambda^{(0,k)}}{(\vB_0^t)_{x_0}(\wB_k)_{x_k}}\allowdisplaybreaks\nonumber\\
& = \frac{1}{\lambda^{(0,k)}}\sum_{x_k}\sum_{x_0}(\vB_0^t)_{x_0}L^{(0,k)}_{x_0x_k}(\wB_k)_{x_k}\log\frac{\lambda^{(0,k)}}{(\wB_k)_{x_k}}\allowdisplaybreaks\nonumber\\
& = \sum_{x_k}(\vB_k^t)_{x_k}(\wB_k)_{x_k}\log\frac{\lambda^{(0,k)}}{(\wB_k)_{x_k}} = \sum_{x_k}(\pi_k)_{x_k}\log\frac{\lambda^{(0,k)}(\vB_k^t)_{x_k}}{(\pi_k)_{x_k}}\allowdisplaybreaks\nonumber\\& 
= \sum_{x_k}(\pi_k)_{x_k}\log\frac{\sum_{x_0}L^{(0,k)}_{x_0x_k}}{(\pi_k)_{x_k}}
= H(\pi_k) + \sum_{\beta}(\pi_k)_{\beta}\log\sum_{\alpha}L^{(0,k)}_{\alpha\beta},\label{abovecomp}%
\end{align}
where $H(\pi_k)$ is the entropy of the probability vector $\pi_k$. Corollary \ref{cor_nsft_me} yields%
\begin{equation*}
  h(\sigma_{\infty},\mu_{\infty}) = \lim_{\ep\searrow0}\limsup_{k\rightarrow\infty}\frac{1}{k}H_{\mu_0}\Big(\bigvee_{i=0}^k \sigma_0^{-i}\PC_{i,\ep}\Big).%
\end{equation*}
Now fix an $\ep>0$ and write $n_k := k + m_k(\ep)$. An element of $\bigvee_{i=0}^k \sigma_0^{-i}\PC_{i,\ep}$ has the form $\bigcap_{i=0}^k \sigma_0^{-i}[.x_i^{(i)} \ldots x_{i + m_i(\ep)}^{(i)}]$. A sequence $x$ is contained in this set iff $\sigma_0^i(x) \in [.x_i^{(i)} \ldots x_{i + m_i(\ep)}^{(i)}]$ for $i = 0,1,\ldots,k$. This is equivalent to $x_{i + j} = [\sigma_0^i(x)]_j = x_{i + j}^{(i)}$ for $0 \leq i \leq k,\ 0 \leq j \leq m_i(\ep)$. But this implies%
\begin{equation*}
  \bigvee_{i=0}^k \sigma_0^{-i}\PC_{i,\ep} = \PC_0^{\widetilde{n}_k}.%
\end{equation*}
Replacing $k$ with $\widetilde{n}_k(\ep)$ in \eqref{abovecomp} and observing that $0 \leq H(\pi_k) \leq \log\sup_{i\geq0}l_i$ for all $k$, formula \eqref{eq_me_form} follows.%

(ii) To prove the equality $h(\sigma_{\infty},\mu_{\infty}) = h_{\tp}(\sigma_{\infty})$ under the given condition, first observe that $\widetilde{n}_k(\ep) \geq k + m_k(\ep)$ implies that%
\begin{equation*}
  \bigvee_{i=0}^k \sigma_0^{-i}\PC_{i,\ep} \mbox{\quad is finer than\quad } \PC_0^{n_k(\ep)}.%
\end{equation*}
Consequently, $H_{\mu_0}(\bigvee_{i=0}^k\sigma_0^{-i}\PC_{i,\ep}) \geq H_{\mu_0}(\PC_0^{n_k(\ep)})$ and, after replacing $\widetilde{n}_k(\ep)$ by $n_k(\ep)$ in \eqref{eq_me_form}, the right-hand side is still a lower bound for the entropy. For a given $\delta\in(0,1)$ we choose $\ep>0$ according to the assumption. We can write%
\begin{equation*}
  \sum_{\alpha}L^{(0,n_k(\ep))}_{\alpha\beta} = \sum_{\alpha,\gamma}L^{(0,k)}_{\alpha\gamma}L^{(k,n_k(\ep))}_{\gamma\beta}.%
\end{equation*}
Choosing $\gamma'$ that maximizes the sum $\sum_{\alpha}L^{(0,k)}_{\alpha\gamma}$, our assumption gives%
\begin{equation*}
  \sum_{\alpha}L^{(0,n_k(\ep))}_{\alpha\beta} \geq \sum_{\alpha}L^{(0,k)}_{\alpha\gamma'} \mbox{\quad for all\ } \beta \in B,%
\end{equation*}
where $B$ is a set with $\sum_{\beta\in B}(\pi_{n_k(\ep)})_{\beta} \geq 1 - \delta$. We have $\sum_{\alpha}L^{(0,k)}_{\alpha\gamma'} = \|L^{(0,k)}\|_1$, where $\|\cdot\|_1$ is the operator norm derived from the $\ell^1$-vector norm. Hence,%
\begin{align*}
  h(\sigma_{\infty},\mu_{\infty}) &\geq \limsup_{k\rightarrow\infty}\frac{1}{k}\sum_{\beta \in B}(\pi_{n_k(\ep)})_{\beta}\log\big\|L^{(0,k)}\big\|_1\\
	                                &\geq (1-\delta)\limsup_{k\rightarrow\infty}\frac{1}{k}\log\big\|L^{(0,k)}\big\|_1\\
																	   &= (1-\delta)\limsup_{k\rightarrow\infty}\frac{1}{k}\log\big\|L^{(0,k)}\big\|_{(k)},%
\end{align*}
where we use that all norms on a finite-dimensional space are equivalent and the equivalence factors stay bounded, when the dimensions vary within a finite set. Letting $\delta \searrow 0$ we obtain $h(\sigma_{\infty},\mu_{\infty}) \geq h_{\tp}(\sigma_{\infty})$, using Theorem \ref{thm_te_formula}. The opposite inequality is satisfied by \eqref{eq_varineq}.%
\end{proof}

\begin{remark}
Condition (ii) in Theorem \ref{thm_nsft_me2} is an asymptotic version of condition \ref{thm_nsft_me1}(ii). If we interpret the NSFT as a time-dependent communication network, this condition (roughly speaking) says that the best receivers of messages are also good broadcasters. We can easily check that in the stationary case an even stronger property is always satisfied. Indeed, let $L$ be a quadratic $l\tm l$ reduced and primitive $0$-$1$-matrix. Pick $p$ such that all entries of $L^p$ are positive. The sum norm satisfies $\|L^{n+1}\| \geq \|L^n\|$ for all $n$. Now take $\ep>0$ with $\|L^p\| < 1/\ep$. Hence, if $w(0,m) > 1/\ep$, then $\|L^m\| = w(0,m) > \|L^p\|$, and consequently $m > p$, implying that all entries of $L^m$ are positive.%
\end{remark}

Finally, we present an example of an NSFT satisfying $h_{\tp}(\sigma_{\infty}) = h(\sigma_{\infty},\mu_{\infty})$, which violates assumption (iii) in \eqref{thm_nsft_me1}. Hence, we see that this assumption is only sufficient, but not necessary.%

\begin{example}
We consider a sequence $(L_i)_{i=0}^{\infty}$ of the form%
\begin{equation*}
  A [B \cdots B]_{k_1} A [B \cdots B]_{k_2} A [B \cdots B]_{k_3} A \ldots,%
\end{equation*}
where $[B \cdots B]_k$ stands for a finite sequence of $k$ copies of $B$, $(k_i)_{i=1}^{\infty}$ is an unbounded sequence of positive integers, and%
\begin{equation*}
  A = \left(\begin{array}{ccc}
	                 1 & 1 & 1 \\
									 1 & 1 & 1 \\
									 1 & 1 & 1
						\end{array}\right),\quad B = \left(\begin{array}{ccc}
						                                    1 & 1 & 0 \\
																								1 & 1 & 0 \\
																								0 & 0 & 1 \end{array}\right).%
\end{equation*}
To compute the relevant quantities for the associated NSFT, first observe that%
\begin{equation*}
  B^k = \left(\begin{array}{ccc}
	                2^{k-1} & 2^{k-1} & 0 \\
				 				  2^{k-1} & 2^{k-1} & 0 \\
								  0 & 0 & 1
						  \end{array}\right),\quad A B^k = \left(\begin{array}{ccc}
	                 2^k & 2^k & 1 \\
									 2^k & 2^k & 1 \\
									 2^k & 2^k & 1
						  \end{array}\right).%
\end{equation*}
By induction, one proves that%
\begin{equation*}
  \left(\begin{array}{ccc}
	                2^{k_1} & 2^{k_1} & 1 \\
				 				  2^{k_1} & 2^{k_1} & 1 \\
								  2^{k_1} & 2^{k_1} & 1
						  \end{array}\right)
    \cdots
							  \left(\begin{array}{ccc}
	                2^{k_n} & 2^{k_n} & 1 \\
				 				  2^{k_n} & 2^{k_n} & 1 \\
								  2^{k_n} & 2^{k_n} & 1
						  \end{array}\right) = \prod_{i=1}^{n-1}\left(1 + 2^{k_i+1}\right)\left(\begin{array}{ccc}
	         2^{k_n} & 2^{k_n} & 1\\
					 2^{k_n} & 2^{k_n} & 1\\
					 2^{k_n} & 2^{k_n} & 1
				\end{array}\right).%
\end{equation*}
\emph{Computation of the norm growth rate:} From the assumption that $(k_i)_{i=1}^{\infty}$ is unbounded it follows that $f(n) = \sum_{i=1}^n k_i$ grows faster than any linear function. Together with the estimates $2^{k_i+1} \leq 1 + 2^{k_i+1} \leq 2^{k_i+2}$, we obtain%
\begin{equation*}
  \lim_{n\rightarrow\infty}\frac{1}{n + \sum_{i=1}^n k_i}\log 2^{k_n}\prod_{i=1}^{n-1} (1 + 2^{k_i+1}) = \log 2.%
\end{equation*}
To see that the growth rate on any other subsequence $(l_n)$ cannot be larger than $\log 2$, assume%
\begin{equation*}
  L^{(0,l_n)} = (AB^{k_1}) \cdots (AB^{k_{m(n)}})(AB^{r_n}), \quad 0 \leq r_n \leq k_{m(n)+1}.%
\end{equation*}
The associated growth rate satisfies%
\begin{equation*}
  \limsup_{n\rightarrow\infty}\frac{\log 2^{r_n}\prod_{i=1}^{m(n)} (1 + 2^{k_i+1})}{m(n) + \sum_{i=1}^{m(n)}k_i + 1 + r_n} \leq \log 2 \cdot \limsup_{n\rightarrow\infty}\frac{(\sum_{i=1}^{m(n)}k_i + r_n) + 2m(n)}{(\sum_{i=1}^{m(n)}k_i + r_n) + m(n) + 1}.%
\end{equation*}
Since $\sum_{i=1}^{m(n)}k_i$ grows faster than linear, we see that $\log 2$ is an upper bound for the above expression, and hence the growth rate of $\|L^{(0,n)}\|$ is $\log 2$.%

\emph{Computation of the eigenvalue growth rate:} The vector $\wB_i$ is an element of $\bigcap_n L^{(i,i+n)}C^+_{i+n}$ (cf.~\cite[Lem.~4.2]{Fis}) In the case $L_i = A$, i.e., $i = 0,k_1+1,k_1+k_2+2,\ldots$, this intersection reduces to the ray $\{\lambda \cdot (1,1,1)^T : \lambda \geq 0\}$, and hence%
\begin{equation*}
  \wB_i = \frac{1}{3}\left(\begin{array}{c} 1 \\ 1 \\ 1 \end{array}\right),\quad i = \sum_{j=1}^n k_j + n,\quad n\in\N.%
\end{equation*}
The equality $L^{(i-k_n,i)}\wB_i = \lambda^{(i-k_n,i)}\wB_{i-k_n}$ implies%
\begin{equation*}
 \frac{1}{3}\left(\begin{array}{c} 2^{k_n} \\ 2^{k_n} \\ 1 \end{array}\right) = \frac{1}{3}B^{k_n}\left(\begin{array}{c} 1 \\ 1 \\ 1 \end{array}\right) = \lambda^{(\sum_{j=1}^{n-1} k_j + n, \sum_{j=1}^nk_j + n)}\wB_{\sum_{j=1}^{n-1} k_j + n}.%
\end{equation*}
Taking the norm on both sides gives%
\begin{equation}\label{eq_lambda_finiteprod}
  \lambda^{(\sum_{j=1}^{n-1} k_j + n, \sum_{j=1}^nk_j + n)} = \frac{1}{3}\left(2^{k_n+1} + 1\right).%
\end{equation}
To compute the growth rate of $\lambda^{(0,n)}$, we also determine the numbers $\lambda_0,\lambda_{k_1+1},\lambda_{k_1+k_2+2},\ldots$. From $L_i\wB_{i+1} = \lambda_i\wB_i$ we get%
\begin{equation*}
  A\wB_{\sum_{j=1}^n k_j + n + 1} = \frac{1}{3}\lambda_{\sum_{j=1}^n k_j + n}\left(\begin{array}{c} 1 \\ 1 \\ 1 \end{array}\right).%
\end{equation*}
Looking at this equation componentwise, we see that $1 = \|\wB_{\sum_{j=1}^n k_j + n + 1}\| = (1/3)\lambda_{\sum_{j=1}^n k_j + n}$, and hence $\lambda_{\sum_{j=1}^n k_j + n} = 3$. Consequently, for $l(n) = \sum_{j=1}^n k_j + n$ we obtain $\lambda^{(0,l(n))} = \prod_{i=1}^n\left(2^{k_i+1} + 1\right)$. Since $\sum_{i=1}^n k_i$ grows faster than any linear function and $2^{k_i+1} \leq 2^{k_i+1} + 1 \leq 2^{k_i+2}$, we obtain%
\begin{eqnarray*}
  \lim_{n\rightarrow\infty}\frac{1}{l(n)}\log\lambda^{(0,l(n))} &=& \lim_{n\rightarrow\infty}\frac{1}{l(n)}\sum_{i=1}^n \log (2^{k_i + 1} + 1)\\
	&=& \lim_{n\rightarrow\infty}\frac{1}{l(n)}\sum_{i=1}^n (k_i+1)\log 2 = \log 2.%
\end{eqnarray*}
With the same reasoning as in the case of the norm growth rate, one shows that $\log 2$ is actually equal to $\limsup_n (1/n)\log\lambda^{(0,n)}$.%

\emph{Assumption (iii) in Theorem 4.13 is violated if $k_i = 2^i$:} Observing that for $i = n + \sum_{j=1}^n k_j + 1$ we have $N_i = k_{n+1}+1 = 2^{n+1} + 1$, and using \eqref{eq_lambda_finiteprod}, we obtain%
\begin{eqnarray*}
  \lambda^{(i,i+N_i)} &=& \lambda^{(\sum_{j=1}^n k_j + n + 1,\sum_{j=1}^{n+1} k_j + (n + 1) + 1)}\\
	&=& \frac{\lambda_{\sum_{j=1}^{n+1}k_j + (n+1)+1}}{\lambda_{\sum_{j=1}^n k_j + n}}\lambda^{(\sum_{j=1}^n k_j + n,\sum_{j=1}^{n+1} k_j + (n+1))}	\geq c \cdot \left(2^{k_{n+1}+1} + 1\right)%
\end{eqnarray*}
with a constant $c>0$. Now $\limsup_{i\rightarrow\infty}\frac{1}{i}\log\lambda^{(i,i+N_i)} \geq \log 2$ follows easily.%
\end{example}

\section{Open questions}%

We end with a list of open problems that we were not able to solve and we think are interesting for further research.%
\begin{enumerate}
\item[(1)] Does there exist any IMS which is \emph{not} a fine-scale IMS for a topological NDS with uniformly totally bounded state space?%
\item[(2)] Are there easily checkable conditions under which condition (ii) in Theorem \ref{thm_constcond} is satisfied?%
\item[(3)] Under which conditions can an IMS be transformed in the sense of Corollary \ref{cor_countablelimitset} to obtain a countable number of non-equivalent limit points?%
\item[(4)] Are there examples of NSFT's, for which the topological entropy is strictly larger than the entropy of any Parry IMS? If so, do these systems still satisfy the full variational principle?
\end{enumerate}

\end{document}